\documentclass[a4paper,12pt]{article}

\usepackage{stmaryrd}
\usepackage{hyperref}
\usepackage{float}
\usepackage{graphicx}
\usepackage{enumerate}
\usepackage{amsthm}
\usepackage[leqno]{amsmath}
\usepackage{amsfonts}
\usepackage{dsfont}
\usepackage{amssymb}

\def\Pas{\hbox{$P$-a.s.}}
\newcommand{\Ind}{I}

\graphicspath{ {images/} }
\numberwithin{equation}{section}

\title{Characterisation of $L^0$-boundedness 
	\\ for a general set of processes 
	\\ with no strictly positive element}
 
\author{
	D\'aniel \'Agoston B\'alint  \\ 
	\vspace{-0.5cm}\\
	Department of Mathematics, ETH Zurich,\\
	R\"amistrasse 101, CH-8092 Zurich, Switzerland \\
	{\tt daniel.balint@math.ethz.ch}
	\\
	\date{}
}

\begin{document}

\openup1\jot

\maketitle
\ \vskip-3\baselineskip
\ 
\begin{abstract}
	\noindent
	We consider a general set $\mathcal{X}$ of adapted nonnegative stochastic processes in infinite continuous time. $\mathcal{X}$ is assumed to satisfy mild convexity conditions, but in contrast to earlier papers need not contain a strictly positive process. We introduce two boundedness conditions on $\mathcal{X}$ --- DSV corresponds to an asymptotic $L^0$-boundedness at the first time all processes in $\mathcal{X}$ vanish, whereas NUPBR$_{\rm loc}$ states that $\mathcal{X}_t = \{ X_t : X \in \mathcal{X}\}$ is bounded in $L^0$ for each $t \in [0,\infty)$. We show that both conditions are equivalent to the existence of a strictly positive adapted process $Y$ such that $XY$ is a supermartingale for all $X \in \mathcal{X}$, with an additional asymptotic strict positivity property for $Y$ in the case of DSV.
\end{abstract}

\vskip1\baselineskip
\noindent {\bf Mathematics subject classification (2010): } 60G48, 91B02, 91G99
\\
%\noindent {\bf JEL classification: } C00, G10
%\\
\noindent {\bf Keywords: } $L^0$-boundedness, supermartingale, NUPBR, viability, set of wealth processes, absence of numeraire, fundamental theorem of asset pricing

\begin{center}
	\ \vskip0\baselineskip
	
	This version: \today
	
	\ \vskip-1\baselineskip
\end{center}

\newpage

\newtheorem{theorem}{Theorem}[section]
\newtheorem{ass}{Assumption}[]
\newtheorem{lemma}[theorem]{Lemma}
\newtheorem{cor}[theorem]{Corollary}
\newtheorem{prop}[theorem]{Proposition}
\theoremstyle{definition}
\newtheorem{example}[theorem]{Example}
\newtheorem{remark}[theorem]{Remark}
\newtheorem{definition}[theorem]{Definition}

\newenvironment{enume}{
	\renewcommand*{\theenumi}{B$'$}
	\renewcommand*{\labelenumi}{(B$'$)}
	\enumerate
}{
	\endenumerate
}

%%%%%%%%%%%%%%%%%%%%%%%%%%%%%%%%%%%%%%%%%%%%%%%%%%%%%%%%%%%%%%%%%%%%%%%%%%%%%%%%%%%%%%%%%%%%%%%%

\section{Introduction}

Consider a general set $\mathcal{X}$ of nonnegative adapted stochastic processes on some time interval $\mathcal{I}$. The set $\mathcal{X}$ is assumed to be convex and to satisfy a so-called switching property which can be understood as a time- and $\omega$-dependent convexity property. We investigate when $\mathcal{X}$ obtains a \lq\lq best element\rq\rq{} $X^* \in \mathcal{X}$ in the sense that $X/X^*$ is well defined and a supermartingale for all $X \in \mathcal{X}$. More generally, we ask, under which conditions on $\mathcal{X}$ there exists a strictly positive $Y$ such that $XY$ is a supermartingale for all $X \in \mathcal{X}$. (Of course, if $X^* :=1/Y$ is an element of $\mathcal{X}$, it is then a best element of $\mathcal{X}$.) We manage to prove results of the following type: (i) there exists a $Y$ as above if and only if (ii) $\mathcal{X}$ satisfies some $L^0$-boundedness condition. Condition (i) is then called a \textit{dual characterisation} of the boundedness condition in (ii).

Similar questions in a less general framework where $\mathcal{X}$ consists of stochastic integrals against a fixed semimartingale integrator are common in the literature (see e.g. \cite{KK07, TS14, KKS16}). General sets $\mathcal{X}$ as above were first considered by Kabanov \cite{Kab97} and used later, among others, by Kardaras \cite{Kar13d}. All the abovementioned works crucially assume that 
\begin{equation}\label{ass}
	\text{there exists a strictly positive process $\bar{X} \in \mathcal{X}$.}
\end{equation}
Under \eqref{ass} and on a finite time interval $\mathcal{I}=[0,T]$, Kardaras \cite{Kar13d} establishes the following result: $\mathcal{X}_T = \{ X_T : X \in \mathcal{X}\}$ is bounded in $L^0$ if and only if there exists a strictly positive process $Y$ such that $XY$ is a supermartingale for every $X \in \mathcal{X}$. The above boundedness condition is known as \textit{no unbounded profit with bounded risk (NUPBR)} in the mathematical finance literature. The background of this result is that since $\mathcal{X}_T$ is convex and $L^0$-bounded, its $L^0$-closure $\bar{\mathcal{X}}_T$ is convex-compact (see \v{Z}itkovi\'{c} \cite{Zit10}) and hence has a maximal element $X^*_T$ in some sense. Due to \eqref{ass} and the switching property, the $L^0$-boundedness can be transferred backwards to $\mathcal{X}_t$ for each $t \in [0,T]$, and hence the same argumentation yields a maximal element $X^*_t$ for each $\bar{\mathcal{X}}_t$. From \eqref{ass} follows moreover that each $X^*_t$ is strictly positive. Then one carefully pastes the $1/X^*_t$ together, using again \eqref{ass}, to obtain a process $Y$ which is in turn as in (i). Note how \eqref{ass} is used not only to formulate the problem (because it allow to divide by $\bar{X}$), but also appears in several places in the proofs.

We extend this result to an infinite horizon, i.e., $\mathcal{I} = [0,\infty)$, and, most importantly, remove the assumption \eqref{ass}. For this, we use two distinct approaches. First, we adapt a boundedness condition from B\'{a}lint/Schweizer \cite[Definition~2.7]{BS18} to our framework. This condition is called \textit{dynamic share viability (DSV)} and is a generalisation of NUPBR. Second, we adopt another generalisation of NUPBR, namely the local version NUPBR$_{\rm loc}$  introduced by Chau et al. \cite{CCFM17}. Both DSV and NUPBR$_{\rm loc}$ can be applied (in fact, are designed) for the infinite-horizon framework.

If \eqref{ass} is not assumed, then each $X \in \mathcal{X}$ can reach zero. In particular, it is possible that all $X\in\mathcal{X}$ are zero at a certain random time $\hat{T}$. If this does not happen for $\omega$, we can set $\hat{T}(\omega)=\infty$. These considerations connect our two objectives in a subtle way: working with an infinite time horizon (or, more generally, on a right-open time interval), and removing \eqref{ass}. This connection becomes clearer if we think of $\hat{T}$ as the right end of a right-open random time interval. On $\llbracket \hat{T}, \infty \rrbracket$, the processes $X\in\mathcal{X}$ become and presumably should stay zero, while on $\llbracket 0, \hat{T} \llbracket$, we expect to have some processes which are strictly positive. The idea is then to focus on the interval $\llbracket 0 , \hat{T} \llbracket$; we create from $\mathcal{X}$ an auxiliary set $\mathcal{Z}$ of processes by setting each $X \in \mathcal{X}$ on $\llbracket \hat{T}, \infty \rrbracket$ equal to its left $\liminf$ at $\hat{T}$ and then taking convex and switching-closures. If there exists a process $\hat{X}$ which is strictly positive on $\llbracket 0, \hat{T} \llbracket$, then up to a few technicalities, we may assume that \eqref{ass} holds for $\mathcal{Z}$, which enables us to apply to $\mathcal{Z}$ the results and techniques from \cite{Kar13d}. A careful design of $\mathcal{Z}$ allows us to subsequently transfer the conclusions for $\mathcal{Z}$ to the original set $\mathcal{X}$. With this, we are able to establish our first main result in Theorem~\ref{thm 1}. This is a dual characterisation of DSV by the existence of an adapted process $Y$ strictly positive on $\llbracket 0, \hat{T} \llbracket$, with $\liminf_{t \upuparrows \hat{T}} \hat{X}_t Y_t >0$, and such that $XY$ is a supermartingale for each $X \in \mathcal{X}$. Note that the existence of $\hat{X}$ is considerably less restrictive than \eqref{ass}.

To deal with the most general case where \eqref{ass} is removed entirely, we show that if $0$ is absorbing for all processes in $\mathcal{X}$, then there is a sequence $(\hat{T}^n)_{n \in \mathbb{N}}$ of random times which increases to $\hat{T}$ and such that for each $n \in \mathbb{N}$, there exists a process $\hat{X}^n$ which is strictly positive on $\llbracket 0, \hat{T}^n \llbracket$. We then use a localised version of the techniques above and paste together the resulting dual processes. This involves both enlarging our filtration and projecting back after doing a construction. Our second main result Theorem~\ref{thm 2} in the general case without \eqref{ass} is then a dual characterisation of NUPBR$_{\rm loc}$ by the existence of a strictly positive adapted process $Y$ such that $XY$ is a supermartingale for each $X \in \mathcal{X}$.

The rest of the paper is structured as follows. In Section~1.1, we introduce notation. Section~1.2 defines and motivates our $L^0$-boundedness conditions DSV and NUPBR$_{\rm loc}$, Section~1.3 presents our main results, and Section~1.4 reviews the literature. The somewhat technical proofs are given in Section~2. Finally, Section~3 contains some brief comments and remarks on the case where the processes in $\mathcal{X}$ are not assumed to be adapted, the so-called non-adapted framework.

%%%%%%%%%%%%%%%%%%%%%%%%%%%%%%%%%%%%%%%%%%%%%%%%%%%%%%%%%%%%%%%%%%%%%%%%%%%%%%%%%%%%%%%%%%%%%%%%

\subsection{Notation}

We work on a filtered probability space $(\Omega, \mathcal{F}, \mathbb{F}, P)$ with the filtration $\mathbb{F} = (\mathcal{F}_t)_{t \geq 0}$ satisfying the usual conditions. We assume that $\mathcal{F}_0$ is trivial. Throughout this paper, we use the convention $0/0 = 1$. For any set $\mathcal{X}$ of stochastic processes on $[0,\infty)$ and any random variable $\tau: \Omega \rightarrow [0,\infty]$, we write $\mathcal{X}_\tau := \{X_\tau : X \in  \mathcal{X}\}$ (where one, of course, must make sure that $X_\infty$ is defined on $\{\tau = \infty \}$). For a process $X$ and $t_0 \in (0,\infty]$, we denote by $\liminf_{t \upuparrows t_0} X_t$ the left limes inferior of $X$ at $t_0$ such that $t <t_0$. For any random time \mbox{$\tau : \Omega  \rightarrow [0,\infty]$}, we define the stochastic interval \mbox{$\llbracket 0, \tau \rrbracket := \{  (\omega, t) \in \Omega \times [0,\infty): 0 \leq t \leq \tau(\omega) \}$}, and analogously for \mbox{$\llbracket 0, \tau \llbracket$, $\llbracket \tau, \infty \rrbracket$}, etc. Note that every stochastic interval is a subset of $\Omega \times [0,\infty)$ so that $\llbracket \tau, \infty \rrbracket = \llbracket \tau , \infty \llbracket$.

\newpage

\begin{definition}\label{def}
	For a set $\mathcal{X}$ of stochastic processes, we consider the following conditions.
	\begin{enumerate}[$\rm (A)$]
		\item \label{a} Each $X \in \mathcal{X}$ is an adapted nonnegative \Pas{} RC process on $[0,\infty)$ with $X_0 = 1$.
		\item \label{b} There exists a strictly positive process $\bar{X} \in \mathcal{X}$.
		\item \label{c} For any $X, X' \in \mathcal{X}$ and $\alpha \in [0,1]$, the process $cc^\alpha (X, X') := (1-\alpha)X + \alpha X'$ given by the convex combination operator $cc$ is an element of $\mathcal{X}$.
		\item \label{d} For all $t \in [0,\infty)$ and $A \in \mathcal{F}_t$, all $X \in \mathcal{X}$ and all $X' \in \mathcal{X}$ with $\{X_t' = 0 \} \subseteq \{X_t = 0 \}$ on $A$,
		%		\mbox{$\{X'  = 0\} \subseteq \{X = 0\}$}
		the process $sw^{t, A}(X, X') := \Ind_{A^c} X_\cdot + \Ind_{A} ({X_{t \wedge \cdot}}/{X'_t}) X'_{t \vee \cdot}$ given by the switching operator $sw$ is an element of $\mathcal{X}$.
	\end{enumerate}
	Consider moreover the following weakening of \eqref{b}:
	\begin{enume}
		\item \label{B} There exists an $\hat{X} \in \mathcal{X}$ with $\{\hat{X}=0\} \subseteq \{X = 0\}$ for all $X \in \mathcal{X}$.		
	\end{enume}
	We call $\mathcal{X}$ a \textit{set of processes (SP)} if it satisfies \eqref{a}, \eqref{c} and \eqref{d}; a \textit{set of processes with a dominating process (SPD)} if it satisfies in addition \eqref{B}; and a \textit{set of processes with a strictly positive process (SPP)} if it satisfies in addition \eqref{b}.
\end{definition}

In Definition~\ref{def}, every $X \in \mathcal{X}$ is defined on the infinite time interval $[0,\infty)$, normalised at $t=0$ and has some regularity properties, most notably, nonnegativity, cf. \eqref{a}. The set $\mathcal{X}$ is further assumed to be convex, see \eqref{c}, and closed under switching between processes, see \eqref{d}. In intuitive terms, the switching operator $sw^{t, A}(X, X')$ describes that we start following a process $X$ at $t=0$, and at time $t$ and if $A$ happens, we shift the entire current value $X_t$ into the alternative process $X'$ and continue following $X'$. Condition \eqref{d} is sometimes also called fork-convexity; it has been first introduced by \v{Z}itkovi\'{c}, \cite[Definition~6]{Zit02}, see also \cite{Zit10} and Remark~\ref{rmk 5}, and is motivated from mathematical finance, see Remark~\ref{rmk mf}. For a set $\mathcal{X}$ as above, we may impose the existence of either a strictly positive process $\bar{X}$, see \eqref{b}, or less restrictively, a dominating process $\hat{X}$, see \eqref{B}, or completely refrain from making any assumptions on the existence of some kind of strictly positive process. Example~\ref{ex 2} illustrates all three possibilities.

\begin{remark}[On switching]\label{rmk 5}
	Sets of processes similar to Definition~\ref{def} have first been introduced in Kabanov \cite{Kab97}, and were later studied, among others, by Kardaras \cite{Kar13d}. Crucially, both assume the existence of a strictly positive process as in \eqref{b}. Definition~\ref{def} extends these concepts to a setting where \eqref{b} does not hold, cf. SPD and SP. For this, we need to strengthen the switching property \eqref{d} in a way that accounts for the lack of a strictly positive process. To provide some context, the switching property (d) in \cite[Theorem~2.3]{Kar13d} is like \eqref{d}, but with the difference that it requires $X'>0$ instead of the inclusion \mbox{$\{X_t'  = 0\} \subseteq \{X_t = 0\}$} on $A$. Hence \eqref{d} is indeed a strengthening of (d) in \cite[Theorem~2.3]{Kar13d}. Working with \eqref{d} is a very natural choice because SP and SPD do not assume the existence of any $X'>0$. Therefore condition (d) in \cite{Kar13d} becomes void in the context of SPDs and SPs, whereas \eqref{d} is meaningful and a straightforward replacement of it.
\end{remark}

\begin{remark}[Interpretation in terms of mathematical finance]\label{rmk mf}
	A set $\mathcal{X}$ of processes as in Definition~\ref{def} can be used to model financial markets, and more precisely, to describe a set of wealth processes. This is a central modelling object in mathematical finance which describes the underlying stochastic mechanisms of a financial market. Every $X \in \mathcal{X}$ models the value evolution of a certain asset (which might consist of combinations of some basic underlying assets or value processes of investment strategies). This interpretation was used by Kabanov \cite{Kab97} and Kardaras \cite{Kar13d}.
	
	The conditions in Definition~\ref{def} can then be interpreted as follows. According to \eqref{a}, the wealth processes are nonnegative, hence either describe primary assets or value processes of general investment strategies with a solvency constraint. Moreover, they are adapted, meaning that their value only relies on available information (this is a very common assumption). The process $\bar{X}$ in \eqref{b} can be interpreted as a numeraire, also called discounter or accounting unit. Condition \eqref{c} demands that an investor is able to allocate his investment to a combination of multiple assets. Finally, the switching operator $sw$ as in \eqref{d} describes in intuitive terms that an investor starts with a wealth process $X$, and at time $t$ and if $A$ happens, she reallocates her entire current wealth $X_t$ into the alternative wealth process $X'$ and then continues holding $X'$.
	
	This setup is more general than (and contains as a special case) the usual approach where the set  $\mathcal{X}$ of wealth processes consists of stochastic integrals against a fixed semimartingale integrator, see e.g. Example~\ref{ex 2}, 1).
\end{remark}

\begin{example}\label{ex 2}
	\textbf{1)} Let $S$ be an $\mathbb{R}_+^N$-valued semimartingale, and consider
	$$
	\Theta := \{ \theta: \text{$\theta$ predictable, $S$-integrable with $\theta_0 \cdot S_0 >0$ and $\theta \cdot S = \theta_0 \cdot S_0 + \theta \bullet S \geq 0$} \},
	$$
	where $\cdot$ denotes the scalar product and $\bullet$ denotes the stochastic integral in the sense of \mbox{\cite[Chapter~III.6]{JacodShiryaev03}}. Define \mbox{$\mathcal{X} := \{ (\theta \cdot S)/(\theta_0 \cdot S_0): \theta \in \Theta \}$} and assume that there exists an $\bar{X} \in \mathcal{X}$ with $\bar{X} >0$ \Pas{} and $\bar{X}_- >0$ \Pas{} This framework is very common in the mathematical finance literature, usually with the following interpretation. $S$ is the price process of $N$ primary (undiscounted) assets, $\Theta$ is the set of all predictable, $0$-admissible, self-financing strategies and $\mathcal{X}$ is the set of time-$0$ normalised wealth processes of these strategies generated from $S$. It is standard to show that $\mathcal{X}$ satisfies \eqref{a}, \eqref{c} and \eqref{d}; moreover, by the existence of $\bar{X}$ as above, $\mathcal{X}$ is even an SPP. 
	
	\textbf{2)} Let $\tau>0$ be any stopping time with $P[\tau \neq \infty] >0$. Then \mbox{$\mathcal{X}^{(\tau)} := \{ X \Ind_{ \llbracket 0, \tau \llbracket} : X \in \mathcal{X} \}$} can be interpreted as a market similar to $\mathcal{X}$ which collapses at some unknown time $\tau$. Clearly, $\mathcal{X}^{(\tau)}$ is an SPD, but not an SPP. 
	
	\textbf{3)} Let $(\tau^n)_{n \in \mathbb{N}}$ be a sequence of stopping times such that $\tau^n < \tau$ for each $n \in \mathbb{N}$ and $\tau^n \upuparrows \tau$ for a random time $\tau >0$ with $P[\tau \neq \infty] >0$. Then \mbox{$\mathcal{X}^{(\tau-)} := \bigcup_{n \in \mathbb{N}} \mathcal{X}^{(\tau^n)}$} clearly is an SP. On the other hand,  $\mathcal{X}^{(\tau-)}$ is not an SPD. Indeed for every $X \in \mathcal{X}^{(\tau-)}$, there exists $n\in\mathbb{N}$ with $X \in \mathcal{X}^{(\tau^n)}$ and due to $\tau^{n+1} \geq \tau^n$, the process \mbox{$X' := \bar{X}\Ind_{ \llbracket 0, \tau^{n+1} \llbracket} \in \mathcal{X}^{(\tau^n)} \subseteq \mathcal{X}^{(\tau-)}$} satisfies $\{X = 0\} \not \subseteq \{ X' =0 \}$.
\end{example}

The following concept comes up later in some proofs.
%definition illustrates explicitly that the bounded-time-horizon framework $[0,T]$ is a genuine special case of the infinite-horizon framework $[0,\infty)$.

\begin{definition}\label{def f}
	A set $\mathcal{X}$ of stochastic processes has \textit{bounded time horizon} if there is a (nonrandom) $T \in (0,\infty)$ such that for each $X\in \mathcal{X}$ and $t\geq T$, $X_t = X_T$.
\end{definition}

\begin{remark}\label{rmk 3}
	With regard to the time horizon, the case $[0,T]$ with $T<\infty$ embeds into and hence is a special case of the case $[0,\infty)$. Mathematically, the crucial property of a continuous time interval is whether it is right-closed (corresponding to the case $[0,T]$) or right-open (corresponding to case $[0,\infty)$). More precisely, any right-closed time interval is isomorphic to $[0,\infty]$ (in the topological sense) since for any $T \in [0,\infty)$, the continuous monotonic time-transformation $f: [0,\infty] \rightarrow [0,T]$, $f(t) := (1 - e^{-t})T$ is a bijection. Similarly, and any right-open time interval is isomorphic to $[0,\infty)$. It follows in particular that the time interval $[0,\infty]$ is a special case of $[0,\infty)$ in the above sense, even though $[0,\infty] \supset [0,\infty)$. This will be used later.
\end{remark}

\begin{definition}\label{def 1}
	For an SPD $\mathcal{X}$ with $\hat{X}$ as in \eqref{B}, define $\hat{T} := \inf\{ t\geq 0: \hat{X}_t = 0 \}$, with the usual convention that $\inf \emptyset = \infty$.
\end{definition}

Note that $\hat{T}$ does not depend on the choice of $\hat{X}$, i.e., for any two $\hat{X}, \hat{X}'$ satisfying \eqref{B}, we have $ \inf\{ t\geq 0: \hat{X}_t = 0 \} = \inf\{ t\geq 0: \hat{X}'_t = 0 \}$. Moreover, if $\mathcal{X}$ is an SPP, then $\hat{T} \equiv \infty$.

%%%%%%%%%%%%%%%%%%%%%%%%%%%%%%%%%%%%%%%%%%%%%%%%%%%%%%%%%%%%%%%%%%%%%%%%%%%%%%%%%%%%%%%%%%%%%%%%

\subsection{Boundedness conditions}

We are looking for stable boundedness conditions for sets of processes as in Definition~\ref{def} which behave well for as general sets as possible, hence ideally also for SPDs and SPs. By good behaviour, we mean that the boundedness condition admits a dual characterisations in terms of martingale properties. The most prominent related result in this framework is due to Kardaras \cite{Kar13d}. In order to recall it, we need to introduce some notation.

\begin{definition}
	A stochastic process $X$ is called \textit{right-continuous (RC) in probability} if for any $t \in [0,\infty)$ and any sequence $(t^n)_{n \in \mathbb{N}}$ with $t^n \downarrow t$ as $n\rightarrow \infty$, we have \mbox{$\lim_{n\rightarrow \infty} X_{t^n} = X_t$} in probability. A stochastic process $X$ is called \textit{right-continuous with left limits (RCLL) in probability} if it is RC in probability and for any $t \in (0,\infty)$ and any sequence $(t^n)_{n \in \mathbb{N}}$ with $t^n \nearrow t$ as $n\rightarrow \infty$, $(X_{t^n})$ converges in probability as $n\rightarrow \infty$.
\end{definition}

\begin{remark}\label{rmk 1}
	We always assume that a nonnegative supermartingale which is RC in probability is \Pas{} RCLL. This is without loss of generality because over a right-continuous filtration, any supermartingale $X$ which is RC in probability has a \Pas{} RCLL version; see Lemma~\ref{lemma 5} for the case where $X$ is nonnegative.
\end{remark}

\begin{definition}
	A process $Y$, not necessarily adapted, is a \textit{supermartingale deflator (SMD)} for an SP $\mathcal{X}$ if $XY$ is a supermartingale for each $X\in\mathcal{X}$.
\end{definition}

\begin{remark}
	Note that we do not a priori require that an SMD is RC or satisfies any other path property. Therefore, for an SMD $Y$ and any $X \in \mathcal{X}$, the process $XY$ need not be RC.
\end{remark}

For the special case where $\mathcal{X}$ is an SPP which has bounded time horizon, Kardaras \mbox{\cite[Theorem~2.3]{Kar13d}} via \cite[Remark~2.4.4]{Kar13d} proves the following result.

\begin{prop}\label{thm kar}
	{\rm(Kardaras \cite[Theorem~2.3]{Kar13d})} Let $\mathcal{X}$ be an SPP which has bounded time horizon $T$. Then $\mathcal{X}_T$ is bounded in $L^0$ if and only if there exists an adapted strictly positive supermartingale deflator $Y$ which is RCLL in probability.
\end{prop}

\begin{remark}
	\cite[Theorem~2.3]{Kar13d} is in fact slightly more general than Proposition~\ref{thm kar}. First, it does not require the processes $X \in \mathcal{X}$ to be adapted. A supermartingale then must be understood in some generalised sense, and the \lq\lq generalised\rq\rq{} supermartingale deflator $Y$ in the dual condition is then of course also not adapted. In Section~\ref{appendix 2} we give some further remarks on this so called \lq\lq non-adapted framework\rq\rq. Second, \cite[Theorem~2.3]{Kar13d} requires a weaker switching property than \eqref{d}, as explained in Remark~\ref{rmk 5}.
\end{remark}

It is immediately apparent from the statement of Proposition~\ref{thm kar} that it cannot be directly extended to the general, infinite-horizon setting of Definition~\ref{def 1} as the boundedness condition is stated for some bounded time horizon $T$. Moreover, Example~\ref{ex 1} below shows that Proposition~\ref{thm kar} is not true if \eqref{b} does not hold, i.e., if $\mathcal{X}$ does not contain a strictly positive element. This in particular means that Proposition~\ref{thm kar} cannot be applied for an SPD (or SP) without further adjustment.

\begin{example}\label{ex 1}
	Let $T \in (0,\infty)$, $\mathcal{X}^0 : = \{ (X_t) : X_t := (1+xt) \Ind_{ \llbracket 0, T \llbracket}(t), x, t \in \mathbb{R}_+ \}$ and define $\mathcal{X}$ as the smallest set containing $\mathcal{X}^0$ such that \eqref{c} and \eqref{d} hold. Then $\mathcal{X}$ clearly satisfies \eqref{a}, \eqref{B} with e.g.~$\hat{X} = \Ind_{ \llbracket 0, T \llbracket}$, \eqref{c} and \eqref{d} (hence is an SPD), but does not satisfy \eqref{b} as all $X \in \mathcal{X}$ are equal to zero at $T$. Also, $\mathcal{X}$ has bounded time horizon $T$ and $\mathcal{X}_T = \{0\}$ is trivially bounded in $L^0$.  On the other hand, there does not exist any strictly positive supermartingale deflator for $\mathcal{X}$. Indeed, observe that the set $\mathcal{Z} := \{ X_{(T/2) \wedge \cdot} : X \in \mathcal{X}\}$ is an SPP with bounded time horizon $T/2$, but $\mathcal{Z}_{T/2} = \mathcal{X}_{T/2}$ is not bounded in $L^0$ and hence there does not exist any strictly positive supermartingale deflator for $\mathcal{Z}$ by Proposition~\ref{thm kar}. Since any strictly positive supermartingale deflator $Y$ for $\mathcal{X}$ gives $Y_{T/2 \wedge \cdot}$ as a strictly positive supermartingale deflator for $\mathcal{Z}$, there cannot exist any strictly positive supermartingale deflator for $\mathcal{X}$ either.
\end{example}

The reason why the above example works is that if $\mathcal{X}$ is not an SPP, the boundedness of $\mathcal{X}_T$ in $L^0$ does not imply any boundedness of $\mathcal{X}_t$ for $t\in (0,T)$. To adjust for this unnatural behaviour, we adapt two boundedness conditions for our generalised setting to be more robust. First, we recall from Chau et al. \cite[Chapter~2.2]{CCFM17} the following definition which is a well-known condition in mathematical finance.

\begin{definition}\label{def nupbr}
	An SP $\mathcal{X}$ satisfies \textit{no unbounded profit with bounded risk locally (NUPBR$_{\rm loc}$)} if for all $t \in [0,\infty)$, $\mathcal{X}_t$ is bounded in $L^0$.
\end{definition}

Lemma~\ref{lemma 2} (b) shows that if $\mathcal{X}$ is an SPP which has bounded time horizon $T$, then $\mathcal{X}$ satisfies NUPBR$_{\rm loc}$ if and only if $\mathcal{X}_T$ is bounded in $L^0$. In particular, Proposition~\ref{thm kar} can be rephrased accordingly. In general, NUPBR$_{\rm loc}$ is of course a stronger condition than the boundedness of $\mathcal{X}_T$ in $L^0$ for some $T \in (0,\infty)$; note that the set $\mathcal{X}$ in Example~\ref{ex 1} does not satisfy NUPBR$_{\rm loc}$ even though $\mathcal{X}_T$ is bounded in $L^0$.

Our second boundedness condition is motivated by B\'alint/Schweizer \cite{BS18}. Whereas Definition~\ref{def nupbr} requires boundedness for each $t \geq 0$ and is therefore a local condition, the idea here is to establish a global condition by controlling the asymptotic behaviour at $\hat{T}$ of the processes $X \in \mathcal{X}$.

\begin{definition}\label{def dsv}
	An SPD $\mathcal{X}$ satisfies \textit{dynamic share viability (DSV) for $\hat{X}$} as in \eqref{B} if 
	$$
	\left\{\liminf_{t \upuparrows \hat{T}} \frac{X_t}{\hat{X}_t}: X \in \mathcal{X} \right\} \text{ is bounded in } L^0.
	$$
\end{definition}

\begin{remark}
	The terminology \lq\lq dynamic share viability\rq\rq{} is also used in \cite{BS18} and has its origins in the mathematical finance literature. In the context of the mathematical finance interpretation given in Remark~\ref{rmk mf}, one understands the terminology as follows. \lq\lq Dynamic\rq\rq{} refers to the fact that investors are able to use dynamically changing investment strategies due to condition \eqref{d}; \lq\lq share\rq\rq{} expresses that we describe the value of an investment $X\in\mathcal{X}$ in shares (in proportions) of a distinguished process $\hat{X}$; and \lq\lq viability\rq\rq{} is a common term to refer to a regularity property of a market modelled by $\mathcal{X}$.
\end{remark}

DSV was introduced in \cite[Definition~2.7]{BS18} for the set $\mathcal{X}$ from Example~\ref{ex 2}, 1) which is a special case of our framework. Definition~\ref{def dsv} is a generalisation of \cite[Definition~2.7]{BS18}; indeed, if $\mathcal{X}$ is as in Example~\ref{ex 2}, 1), the two definitions are equivalent, see Lemma~\ref{lemma 6}.

The concept of DSV also generalises NUPBR$_{\rm loc}$. Indeed, if $\mathcal{X}$ is an SPP which has bounded time horizon and $\bar{X}$ is as in \eqref{b}, then DSV for $\bar{X}$ is equivalent to NUPBR$_{\rm loc}$; see Lemma~\ref{lemma 2} (a) and (c). If $\mathcal{X}$ is only an SPD or does not have bounded time horizon, DSV is strictly stronger than NUPBR$_{\rm loc}$; see Lemma~\ref{lemma 2} (a) and Example~\ref{ex 3}. In general, DSV depends on the choice of $\hat{X}$; in fact, if $\hat{X}^1 \leq \hat{X}^2$, then DSV for $\hat{X}^1$ implies DSV for $\hat{X}^2$, but not necessarily the other way around.

%%%%%%%%%%%%%%%%%%%%%%%%%%%%%%%%%%%%%%%%%%%%%%%%%%%%%%%%%%%%%%%%%%%%%%%%%%%%%%%%%%%%%%%%%%%%%%%%

\subsection{The main results}

We are now able to present our results. The proofs are deferred to Section~\ref{sec prof}. Our first main result is a dual characterization of DSV for $\hat{X}$ in terms of a supermartingale deflator with an additional asymptotic positivity property at \mbox{$\hat{T} := \inf\{ t\geq 0: \hat{X}_t = 0 \}$}.

\begin{theorem}\label{thm 1}
	Let $\mathcal{X}$ be an SPD with $\hat{X}$ as in \eqref{B} and $\hat{T}$ as in Definition~\ref{def 1}. Then the following are equivalent:
	\begin{enumerate}[$\rm (1)$]
		\item $\mathcal{X}$~satisfies DSV for $\hat{X}$.
		\item There exists an SMD $\thinspace Y$ which is adapted, \Pas{} RC, strictly positive on $\llbracket 0, \hat{T} \llbracket$ and satisfies $\liminf_{t \upuparrows \hat{T}} \hat{X}_t Y_t >0$ \Pas
	\end{enumerate}
\end{theorem}

\begin{remark}
	In mathematical finance, a result of the type of \mbox{Theorem~\ref{thm 1}} is referred to as a \textit{fundamental theorem of asset pricing (FTAP)}. Note that both of our main results, Theorem~\ref{thm 1} and Theorem~\ref{thm 2}, are FTAP results.
\end{remark}

We summarize the implications of Theorem~\ref{thm 1} for the special case in \mbox{Example~\ref{ex 2}}, 1). Let $S$, $\Theta$, $\mathcal{X}$ be as in Example~\ref{ex 2}, 1) and let $\eta \in \Theta$ be such that $\eta \geq0$, $\eta \cdot S>0$,  $\eta \cdot S_- >0$ and $S/(\eta \cdot S)$ is bounded uniformly in $t\geq0$, \Pas{}. Then Theorem~\ref{thm 1} via Lemma~\ref{lemma 6} implies the following: $S$ satisfies DSV for $\eta$ if and only if there exists an SMD $Y$ which is adapted, strictly positive and satisfies $\liminf_{t \rightarrow \infty} \big(\eta_t \cdot (S_t Y_t)\big)> 0$ \Pas{} This recovers \mbox{\cite[Theorem~2.11]{BS18}}, and hence Theorem~\ref{thm 1} generalises \cite[Theorem~2.11]{BS18}, also providing an alternative proof for that result.

Under a mild additional assumption, we can say even more. Indeed, if we strengthen \mbox{Theorem~\ref{thm 1}, (1)} by additionally assuming that $0$ is an absorbing state, meaning that \mbox{$\{X_s = 0\} \subseteq \{X_t = 0\}$} for each $X \in \mathcal{X}$ and $0 \leq s < t <\infty$, then we can obtain in (2) an SMD which is strictly positive on the whole interval $[0,\infty)$. Note that the existence of a strictly positive SMD as in (2) already implies that $0$ is an absorbing state; this, as we shall see, is a consequence of Proposition~\ref{prop 1} via Remark~\ref{rmk 2}, 2).

\begin{cor}\label{cor 2}
	Let $\mathcal{X}$ be an SPD with $\hat{X}$ as in \eqref{B} and $\hat{T}$ as in Definition~\ref{def 1}. Then the following are equivalent:
	\begin{enumerate}[$\rm (1)$]
		\item $\mathcal{X}$~satisfies DSV for $\hat{X}$ and $0$ is an absorbing state.
		\item There exists an SMD $\thinspace Y$ which is adapted, \Pas{} RC, strictly positive and satisfies $\liminf_{t \upuparrows \hat{T}} \hat{X}_t Y_t >0$ \Pas
	\end{enumerate}
\end{cor}

Theorem~\ref{thm 1} applies to any SPD, and hence is already more general than Proposition~\ref{thm kar}. To consider the most general case where $\mathcal{X}$ is an SP, we need to develop a better understanding of the general structure of SPs. We start with the following crucial observation: Every SP can be seen as a limit of SPDs; hence SPs and SPDs are \lq\lq not too far away\rq\rq{} from each other. This is made precise in the next proposition; compare Remark~\ref{rmk 2}, 1).

\begin{prop}\label{prop 1}
	Let $\mathcal{X}$ be an SP and consider the following two conditions:
	\begin{enumerate}[$\rm (a)$]
		\item $0$ is an absorbing state.
		\item There exists an SMD $Y$ for $\mathcal{X}$ such that $\{Y=0\} \subseteq \{X= 0\}$ for each $X \in \mathcal{X}$.
	\end{enumerate}
	If either $\rm (a)$ or $\rm (b)$ holds, then there exists a \Pas{} unique stopping time $\widetilde{T}$ such that \mbox{$X = X\Ind_{\llbracket 0, \widetilde{T} \llbracket}$} \Pas{} for all $X \in \mathcal{X}$ and there exists a sequence $(\hat{X}^n)_{n \in \mathbb{N}} \subseteq \mathcal{X}$ with cemetery times $\hat{T}^n := \inf\{t\geq0: \hat{X}^n=0\}$ such that $\hat{T}^n \nearrow \widetilde{T}$ as $n \rightarrow \infty$.
\end{prop}

\begin{remark}\label{rmk 2}
	\textbf{1)} Assume that either (a) or (b) from Proposition~\ref{prop 1} hold and let $\hat{T}^n$ as in the proposition. Then $\mathcal{X}$ can be seen as a limit of $(\mathcal{X}^n)_{n\in\mathbb{N}}$ as $n \rightarrow \infty$ where $\mathcal{X}^n := \mathcal{X} \Ind_{\llbracket 0, \hat{T}^n \llbracket}$ due to $\hat{T}^n \nearrow \widetilde{T}$ and $X = X\Ind_{\llbracket 0, \widetilde{T} \llbracket}$ \Pas{} for all $X \in \mathcal{X}$. Moreover, for every $n\in\mathbb{N}$, $\mathcal{X}^n$ is an SPD, since $\hat{X}^n = \hat{X}^n\Ind_{\llbracket 0, \hat{T}^n \llbracket} \in \mathcal{X}^n$ is positive on $\llbracket 0, \hat{T}^n \llbracket$. Therefore, $\mathcal{X}$ can be seen as a limit of the SPDs $(\mathcal{X}^n)_{n\in\mathbb{N}}$.
	
	\textbf{2)} Due to Proposition~\ref{prop 1} and its proof, we may and do assume without loss of generality by switching between versions that if Proposition~\ref{prop 1} (a) or (b) holds, then for all $X\in \mathcal{X}$, we have $X = X\Ind_{\llbracket 0, \tau^X \wedge \widetilde{T} \llbracket}$ (surely), where $\tau^X := \inf \{ t \geq 0: X_t =0\}$ and $\widetilde{T}$ is defined as in Proposition~\ref{prop 1}. It follows in particular that if (b) holds, we may assume that (a) holds as well. Moreover, under either (a) or (b), the following conditions are equivalent for any process $W$ (not necessarily in $\mathcal{X}$):
	\begin{enumerate}[(1)]
		\item $\{W=0\} \subseteq \{X= 0\}$ for each $X \in \mathcal{X}$.
		\item $W$ is strictly positive on $\llbracket 0, \widetilde{T} \llbracket$.
	\end{enumerate}
	However, perhaps surprisingly, NUPBR$_{\rm loc}$ (and hence DSV) does not imply that $0$ is an absorbing state. In fact, it only implies this if $\widetilde{T}\equiv\infty$; but it might happen that if $\widetilde{T} \not \equiv \infty$, some $X \in \mathcal{X}$ becomes positive after $\widetilde{T}$. For example, if $\mathcal{X}$ is an SPP which satisfies NUPBR$_{\rm loc}$, then \mbox{$\mathcal{X}^{?} : = \{X \Ind_{\llbracket 0,1 \llbracket \cup \llbracket 2,\infty \rrbracket} : X \in \mathcal{X}\}$} is an SPD which satisfies NUPBR$_{\rm loc}$ as well, but $0$ is not an absorbing state for $\mathcal{X}^{?}$.
	
	\textbf{3)} Observe that $\widetilde{T}$ as in Proposition~\ref{prop 1} is a stopping time. Indeed, each $\hat{T}^n$ is a stopping time, and hence so is their monotone limit $\widetilde{T}$. Therefore, $\widetilde{T}$ can only be independent of the filtration $\mathbb{F}$ if it is a constant.
\end{remark}

If $\mathcal{X}$ is an SPD with $\hat{X}$ as in \eqref{B} and $\hat{T}$ as in Definition~\ref{def 1}, then $\widetilde{T}$ as in Proposition~\ref{prop 1} is equal to $\hat{T}$. Therefore, $\widetilde{T}$ can be considered as a generalisation of $\hat{T}$ to the case where $\mathcal{X}$ is an SP, but no SPD. In general, $\widetilde{T}$ plays a similar role for SPs as $\hat{T}$ for SPDs. 

Our second main result is a dual characterization of NUPBR$_{\rm loc}$ in terms of a supermartingale deflator. This result applies to the most general case when $\mathcal{X}$ is an SP.

\begin{theorem}\label{thm 2}
	Let $\mathcal{X}$ be an SP. Then the following are equivalent:
	\begin{enumerate}[$\rm (1)$]
		\item $\mathcal{X}$ satisfies NUPBR$_{\rm loc}$ and $0$ is an absorbing state.
		\item There exists an SMD $Y$ which is adapted and such that $\{Y=0\} \subseteq \{X= 0\}$ for each $X \in \mathcal{X}$.
	\end{enumerate}
\end{theorem}

An immediate consequence of the proof of Theorem~\ref{thm 2} is that NUPBR$_{\rm loc}$ holds only if there exists an SMD $Y$ which is adapted and strictly positive on $\llbracket 0, \widetilde{T} \llbracket$. This is an analogue of Theorem~\ref{thm 1}, \lq\lq (1) $\Rightarrow$ (2)\rq\rq. The converse of this statement does not hold, i.e., the existence of an SMD $Y$ which is adapted and strictly positive on $\llbracket 0, \widetilde{T} \llbracket$ does not imply NUPBR$_{\rm loc}$. This is because the existence of such a $Y$ does not provide any condition which controls $\mathcal{X}$ on the interval $\llbracket \widetilde{T}, \infty \llbracket$; more precisely, if neither Proposition~\ref{prop 1} (a) nor (b) hold, some $X \in \mathcal{X}$ might become positive on $\llbracket \widetilde{T}, \infty \llbracket$ again and violate NUPBR$_{\rm loc}$ (e.g. in a setup like the curious example in Remark~\ref{rmk 2}). On the other hand, we have seen in Remark~\ref{rmk 2} that Theorem~\ref{thm 2}, (2) implies that $0$ is an absorbing state, but NUPBR$_{\rm loc}$ does not. Therefore NUPBR$_{\rm loc}$ without any further assumption does not imply Theorem~\ref{thm 2}, (2). Finally, Theorem~\ref{thm 2}, (2) holds if and only if there exists an SMD $Y'$ which is adapted and strictly positive; indeed, the \lq\lq if\rq\rq{} direction is trivial and the \lq\lq only if\rq\rq{} direction follows by setting $Y' := Y + \Ind_{\{Y=0\}}$ and observing that $Y'$ is still an adapted SMD due to $\{Y=0\} \subseteq \{X= 0\}$ for each $X \in \mathcal{X}$. We summarise these observations in the following corollary.

\begin{cor}\label{cor 1}
	Let $\mathcal{X}$ be an SP and let $\widetilde{T}$ be defined as in Proposition~\ref{prop 1}. Consider the following statements:
	\begin{enumerate}[$\rm (1)$]
		\item There exists an SMD $Y$ which is adapted and strictly positive.
		\item There exists an SMD $Y$ which is adapted and such that $\{Y=0\} \subseteq \{X= 0\}$ for each $X \in \mathcal{X}$.
		\item $\mathcal{X}$ satisfies NUPBR$_{\rm loc}$.
		\item There exists an SMD $Y$ which is adapted and strictly positive on $\llbracket 0, \widetilde{T} \llbracket$.
	\end{enumerate}	
	Then $(1) \Leftrightarrow (2) \Rightarrow (3) \Rightarrow (4)$ and the converses are not true in general. If $0$ is an absorbing state, all four statements are equivalent.
\end{cor}

\begin{remark}
	If $\mathcal{X}$ is as in Example~\ref{ex 2} 1), then Theorem~\ref{thm 2} via Corollary~\ref{cor 1} simplifies to the following: $\mathcal{X}$ satisfies NUPBR$_{\rm loc}$ if and only if there exists an SMD $Y$ which is adapted and strictly positive. This can be seen as the localised version of Karatzas/Kardaras \cite[Theorem~4.12]{KK07}. For a similar dual characterisation, see also Chau et al. \mbox{\cite[Proposition~2.1]{CCFM17}}.
\end{remark}

%%%%%%%%%%%%%%%%%%%%%%%%%%%%%%%%%%%%%%%%%%%%%%%%%%%%%%%%%%%%%%%%%%%%%%%%%%%%%%%%%%%%%%%%%%%%%%%%

\subsection{Comparison to the literature}

The literature on sets of processes with no strictly positive element is rather small. Apart from the following two papers in mathematical finance, we are not aware of other references or related works.

Tehranchi \cite{Teh15} undertakes a study of absence of arbitrage in classic market models in infinite discrete time, i.e., $\mathcal{X}$ consists of stochastic integrals against a fixed semimartingale integrator. He proves (see \cite[Theorem~2.10]{Teh15}) that the absence of so-called investment-consumption arbitrage, which is a scalable type of arbitrage closely related to NUPBR, is equivalent to the existence of a strictly positive process $Y$ such that $XY$ is a martingale for each $X \in \mathcal{X}$.

A recent preprint by Harms et al. \cite{HLN20}, based on research independent of ours, proves some similar results for a generalised set $\mathcal{X}$ of processes like in this paper. In their first two main results \cite[Theorems~2.13 and 2.17]{HLN20}, they prove a dual characterisation of NUPBR$_{\rm loc}$ similar to our second main result in Theorem~\ref{thm 2}. The main difference is that they need two additional assumptions, namely, that there exists an $\hat{X} \in \mathcal{X}$ which is strictly positive on $\llbracket 0 , \hat{T} \llbracket$ (i.e., $\mathcal{X}$ is an SPP) and that $\hat{T}$ takes only countably many values. For comparison, note that we also need the former assumption in our dual characterisation of DSV in Theorem~\ref{thm 1} but not the latter, and we do not need any of them for NUPBR$_{\rm loc}$ in Theorem~\ref{thm 2}. Furthermore, \cite{HLN20} is limited to a finite time horizon; on the other hand, it does not require that the processes $X \in \mathcal{X}$ are adapted.

Concerning the third main result of Harms et al. \cite[Theorem~2.23]{HLN20}\footnote{In the cited version, there are two unclear points in the statement which we assume to be typos: We believe that it should be \lq\lq$P[\tilde{\tau} \in (s, t]] > 0$\rq\rq{} instead of \lq\lq$P[ \thinspace \cdot \thinspace |\tilde{\tau} \in (s, t]] > 0$\rq\rq, and \lq\lq$\tilde{\tau}: \Omega \rightarrow (0,T] \cup \{\infty\}$\rq\rq{} instead of \lq\lq$\tilde{\tau}: \Omega \rightarrow [0,T] \cup \{\infty\}$\rq\rq.}, we found that the technical assumptions used in their statement simplify to the case when $\hat{T} = c$ \Pas{} for a \mbox{$c\in(0,\infty)$}. Indeed, we show in Proposition~\ref{char t} combined with Remark~\ref{rmk 2}, 3) that the random time $\tilde{\tau}$ defined in \cite[Theorem~2.23]{HLN20} is a stopping time, and hence independent of the filtration only if it is \Pas{} constant. This means that \cite[Theorem~2.23]{HLN20} almost reduces to a special case of \cite[Theorem~2.17]{HLN20}; to be precise, the only difference to the special case of \mbox{\cite[Theorem~2.17]{HLN20}} when $\hat{T} = c$ \Pas{} is that \cite[Theorem~2.23]{HLN20} does not require the existence of an $\hat{X} \in \mathcal{X}$ which is strictly positive on $\llbracket 0 , \hat{T} = c \llbracket$, but only that for every $s<c$ there is an $\hat{X}^{(s)} \in \mathcal{X}$ (depending on $s$) which is strictly positive on $[0, s]$. Note that the latter assumption is always satisfied, as we have seen in Proposition~\ref{prop 1}.

%%%%%%%%%%%%%%%%%%%%%%%%%%%%%%%%%%%%%%%%%%%%%%%%%%%%%%%%%%%%%%%%%%%%%%%%%%%%%%%%%%%%%%%%%%%%%%%%

\section{Proofs}
\label{sec prof}

\subsection{Proof of Theorem~\ref{thm 1}}

Before starting with the proof of Theorem~\ref{thm 1}, we briefly outline the main ideas and techniques used for \lq\lq (1)~$\Rightarrow$~(2)\rq\rq, which is the non-standard direction. We start by introducing an auxiliary set $\mathcal{Z}$ of processes which serves as a link between $\mathcal{X}$ and the framework of Kardaras \cite{Kar13d}. We obtain $\mathcal{Z}$ from $\mathcal{X}$ by first discounting by $\hat{X}$, then setting each process equal to its left $\liminf$ at $\hat{T}$ on the interval $\llbracket \hat{T}, \infty \rrbracket$, and finally taking the convex closure and the closure with respect to the switching property. As to the properties of $\mathcal{Z}$, we first argue (Claims 1, 2) that $\mathcal{Z}$ is an SPP (hence satisfies the conditions of \mbox{\cite[Theorem~2.3]{Kar13d})}; then we prove (Claims 3-5) that $\mathcal{Z}_\infty$ is bounded in $L^0$ whenever $\mathcal{X}$ satisfies DSV for $\hat{X}$. The rest of the proof consists of applying \cite[Theorem~2.3]{Kar13d} to $\mathcal{Z}$, modifying the so obtained deflator to be consistent with $\mathcal{X}$, and then carefully checking the required properties.

\begin{proof}[Proof of Theorem~\ref{thm 1}]	
	Consider the set
	\begin{equation}\label{eq z}
		\mathcal{Z}^0 : = \left\{ Z := \frac{X}{\hat{X}} \Ind_{\llbracket 0, \hat{T} \llbracket} + \liminf_{s \upuparrows \hat{T}} \frac{ X_s}{\hat{X}_s} \Ind_{\llbracket \hat{T}, \infty \rrbracket}: X \in \mathcal{X} \right\}
	\end{equation}
	and extend each $Z_\cdot (\omega)$ from $[0,\infty)$ to $[0,\infty]$ by keeping it constant on $[\hat{T}(\omega), \infty]$. Then according to Remark~\ref{rmk 3}, $\mathcal{Z}^0$ is a set of processes with bounded time horizon (since the time horizon is the right-closed interval $[0,\infty]$). In particular, the conditions \eqref{a}--\eqref{d} and \eqref{B} are to be understood on the right-closed interval $[0,\infty]$.	
	
	\textit{Claim 1: $\mathcal{Z}^0$ satisfies \eqref{a} if either $(1)$ or $(2)$ holds.} Note that it is sufficient to show that for all $X\in \mathcal{X}$, $\liminf_{s \upuparrows \hat{T}} { X_s}/{\hat{X}_s} < \infty$ \Pas{} Indeed, from this and $\hat{X}>0$ on $\llbracket 0, \hat{T} \llbracket$ follows that each $Z\in\mathcal{Z}^0$ is a real-valued process, which is adapted, RC and nonnegative by definition and satisfies $Z_0 = X_0/\hat{X}_0 = 1$. If $(1)$ holds, then \mbox{$\{ \liminf_{s \upuparrows \hat{T}} { X_s}/{\hat{X}_s} : X \in \mathcal{X} \}$} is bounded in $L^0$ and $\liminf_{s \upuparrows \hat{T}} { X_s}/{\hat{X}_s}  < \infty$ \Pas{} follows immediately for each $X \in \mathcal{X}$. Suppose now that $(2)$ holds and let $Y$ be a supermartingale deflator as in $(2)$. Fix $X \in \mathcal{X}$; then $XY$ is a supermartingale which is \Pas{} RC since both $X$ and $Y$ are. Define \mbox{$A:=\{ \liminf_{s \upuparrows \hat{T}} X_sY_s = \infty\} \in\mathcal{F}$}. Then for each $n \in \mathbb{N}$, the stopping time $T^n:= \inf\{t\geq 0: X_tY_t \geq n \}$ satisfies $X_{T^n}Y_{T^n}\geq n$ on $A$ by right-continuity of $XY$ and hence $E[X_{T^n}Y_{T^n}]\geq n P[A]$ since $XY$ is nonnegative. On the other hand, $XY$ is right-closable at $\infty$ as a nonnegative supermartingale, and hence the optional stopping theorem applied on $[0,\infty]$ yields for each $n\in \mathbb{N}$ that $E[X_{T^n}Y_{T^n}]\leq X_0 Y_0 = 1$. It follows that $P[A] = 0$ and hence $\liminf_{s \upuparrows \hat{T}}  X_s Y_s  < \infty$ \Pas{} Now since $X$, $Y$ and $\hat{X}$ are all nonnegative and $\liminf_{s \upuparrows \hat{T}} \hat{X}_s Y_s > 0$ \Pas, the contraposition of Lemma~\ref{lemma liminf} (b) with $x(s) := X_s (\omega) /\hat{X}_s(\omega)$ and $y(s) := \hat{X}_s(\omega) Y_s(\omega)$ (for any fixed $\omega \in \Omega$) implies that $\liminf_{s \upuparrows \hat{T}} X_s /\hat{X}_s < \infty$ \Pas{} This proves Claim~1.
	
	For $\alpha \in [0,1]$ and any two processes $Z^1, Z^2$, recall from \eqref{c} the convex combination operator \mbox{$cc^\alpha (Z^1,Z^2) := (1-\alpha) Z^1 + \alpha Z^2$}. Also, for $t \in [0,\infty]$, \mbox{$A \in  \mathcal{F}_t$} and any two processes $Z^1, Z^2$, recall from \eqref{d} the switching operator \mbox{$sw^{t, A}(Z^1,Z^2) := \Ind_{A^c} Z^1_\cdot + \Ind_{A} ({Z^1_{t \wedge \cdot}}/{Z^2_t}) Z^2_{t \vee \cdot}$} if $Z^2$ is a strictly positive process and \mbox{$sw^{t, A}(Z^1,Z^2) := Z^1$} otherwise. Note that by Claim 1, both operators produce real-valued processes when applied to $Z^1, Z^2 \in \mathcal{Z}^0$. Let
	\begin{align*}
		\mathcal{Z}^n 
		:= & \left\{cc^\alpha (Z^1, Z^2) : Z^1, Z^2 \in \mathcal{Z}^{n-1}, \alpha \in[0,1] \right\} 
		\\
		& \cup \left\{sw^{t, A} (Z^1, Z^2) : Z^1, Z^2 \in \mathcal{Z}^{n-1} , t \in [0,\infty], A \in \mathcal{F}_t \right\} \mspace{20mu} \text{for $n \in \mathbb{N}$}
	\end{align*}
	and define the set $\mathcal{Z} := \bigcup_{n \in \mathbb{N}} \mathcal{Z}^n$. 
	
	\textit{Claim 2: $\mathcal{Z}$ is an SPP if either $(1)$ or $(2)$ hold.} Observe that $\mathcal{Z}$ satisfies \eqref{c} and \eqref{d} by construction, and since $\hat{X} \in \mathcal{X}$, $1$ is an element of $\mathcal{Z}^0 \subseteq \mathcal{Z}$, and therefore $\mathcal{Z}$ also satisfies \eqref{b}. Suppose now that either $(1)$ or $(2)$ hold. By induction over $(\mathcal{Z}^n)_{n\in\mathbb{N}}$ and due to Claim  1, it is straightforward to check that $\mathcal{Z}$ satisfies \eqref{a} as well.
	
	\textit{Claim 3: For all $Z \in \mathcal{Z}$,}
	\begin{equation}\label{eq 3}
		Z_s(\omega) = Z_{\hat{T}}(\omega), \mspace{20mu} \forall (\omega, s) \in  \llbracket \hat{T}, \infty \rrbracket.
	\end{equation}
	To prove Claim 3, we use induction over $(\mathcal{Z}^n)_{n\in\mathbb{N}}$. For $n=0$, the claim holds by definition. For the induction step, let $Z \in \mathcal{Z}^n$. Then either $Z = cc^\alpha (Z^1, Z^2)$ or $Z = sw^{t, A}(Z^1, Z^2)$ for some $Z^1, Z^2 \in \mathcal{Z}^{n-1}$ and suitable $\alpha$ respectively $t, A$. In the first case, it follows immediately by the induction hypothesis that $Z$ satisfies \eqref{eq 3}. In the second case, due to the definition of $sw^{t,A}$, we may assume without loss of generality that $Z^2$ is strictly positive and hence \mbox{$Z = \Ind_{A^c} Z^1 _\cdot + \Ind_{A} ({Z^1_{t \wedge \cdot}}/{Z^2_t}) Z^2_{t \vee \cdot}$}. So let $(\omega, s)$ be in $\llbracket \hat{T}, \infty \rrbracket$; then $\hat{T}(\omega)\leq s$ and \mbox{$Z^i_s(\omega) = Z^i_{\hat{T}}(\omega)$} for $i=1,2$ by the induction hypothesis. On $\{t < \hat{T}\}$, we have 
	$$
	Z_s 
	= \Ind_{A^c} Z^1 _{s} + \Ind_{A} ({Z^1_{t \wedge s}}/{Z^2_t}) Z^2_{t \vee \hat{T}} 
	= \Ind_{A^c} Z^1 _{\hat{T}} + \Ind_{A} ({Z^1_{t}}/{Z^2_t}) Z^2_{\hat{T}} 
	= Z_{\hat{T}}.
	$$
	On $\{\hat{T}\leq t \leq s\}$, we have $Z^i_s=Z^i_t = Z^i_{\hat{T}}$ for $i=1,2$ by the induction hypothesis  and hence 
	$$
	Z_s 
	= \Ind_{A^c} Z^1 _{s} + \Ind_{A} ({Z^1_{t \wedge s}}/{Z^2_t}) Z^2_{t \vee \hat{T}}
	= \Ind_{A^c} Z^1_{\hat{T}} + \Ind_{A} ({Z^1_{t}}/{Z^2_{t}}) Z^2_{\hat{T}} 
	= Z_{\hat{T}},
	$$
	whereas on $\{\hat{T}\leq s <t\}$, we simply have $Z_s = Z^1_s = Z^1_{\hat{T}} = Z_{\hat{T}}$. This proves Claim 3.
	
	\textit{Claim 4: For all $Z \in \mathcal{Z}$, there exists an $X \in \mathcal{X}$ such that the following hold:}
	\vspace{-0.06cm}
	\begin{align}\label{eq 4}
		&\frac{X_s(\omega)}{\hat{X}_s(\omega)} = Z_s(\omega), \mspace{20mu} \forall (\omega, s) \in  \llbracket 0, \hat{T}\llbracket.
		\\
		\label{eq 7} &\liminf_{s \upuparrows \hat{T}} \frac{X_s}{\hat{X}_s} \geq Z_{\hat{T}} \mspace{20mu} \Pas 
	\end{align}
	\vspace{-0.06cm}
	For the proof, we again use induction over $(\mathcal{Z}^n)_{n \in \mathbb{N}}$. For $n=0$, the claim holds by definition. For $n\geq1$, let $ Z \in \mathcal{Z}^n$. Then either $Z = cc^\alpha (Z^1, Z^2)$ or \mbox{$Z = sw^{t, A}(Z^1, Z^2)$} for some $Z^1, Z^2 \in \mathcal{Z}^{n-1}$ and suitable $\alpha$ respectively $t, A$. By the induction hypothesis, there exist $X^1, X^2 \in \mathcal{X}$ satisfying \eqref{eq 4} and \eqref{eq 7} with $Z^1, Z^2$ respectively. If $Z = cc^\alpha (Z^1, Z^2)$, define $X : = cc^\alpha(X^1, X^2)$ and note that $X \in \mathcal{X}$ by \eqref{c}. Then by the induction hypothesis, on $\llbracket 0, \hat{T}\llbracket$, 
	\vspace{-0.06cm}
	$$
	\frac{X_s}{\hat{X}_s}
	= (1 - \alpha) \frac{X^1_s}{\hat{X}_s} + \alpha \frac{X^2_s}{\hat{X}_s}
	= (1 - \alpha)Z^1_s + \alpha Z^2_s
	= Z_s.
	$$
	\vspace{-0.06cm}
	Moreover, by the superadditivity of $\liminf$, 
	\vspace{-0.06cm}
	$$
	\liminf_{s \upuparrows \hat{T}} \frac{X_s}{\hat{X}_s}
	\geq \liminf_{s \upuparrows \hat{T}} (1 - \alpha) \frac{X^1_s}{\hat{X}_s}+ \liminf_{s \upuparrows \hat{T}} \alpha \frac{X^2_s}{\hat{X}_s}
	\geq (1 - \alpha) Z^1_T + \alpha Z^2_T 
	= Z_T \mspace{20mu} \Pas
	$$
	\vspace{-0.06cm}
	If $Z = sw^{t, A}(Z^1, Z^2)$, we can assume without loss of generality that $Z^2$ is strictly positive, and hence so is $X^2$ on $\llbracket 0, \hat{T}\llbracket$. Then due to \eqref{d}, $X := sw^{t, A}(X^1, X^2_t)$ is well defined (recall that $0/0 = 1$) and an element of $\mathcal{X}$. By the induction hypothesis and noting the tautology $\hat{X}_s = (\hat{X}_{t \wedge s} \hat{X}_{t \vee s})/\hat{X}_t$, we have on $\llbracket 0, \hat{T}\llbracket$ that
	\vspace{-0.06cm}
	\begin{align*}
		\frac{X_s}{\hat{X}_s}
		= \frac{\Ind_{A^c} X^1_s + \Ind_{A} ({X^1_{t \wedge s}}/{X^2_t}) X^2_{t \vee s}}{\hat{X}_s}
		& = \Ind_{A^c} Z^1_s + \Ind_{A} \frac{X^1_{t \wedge s}}{\hat{X}_{t \wedge s}} \frac{\hat{X}_t}{X^2_t} \frac{X^2_{t \vee s}}{\hat{X}_{t \vee s}}
		\\
		& = \Ind_{A^c} Z^1_s + \Ind_{A} (Z^1_{t \wedge s}/{Z^2_t}) Z^2_{t \vee s}
	\end{align*}
	\vspace{-0.06cm}
	Moreover, using first the definition of $X$, second a distinction of cases into $A^c$, \mbox{$A \cap \{ \hat{T} \leq t\}$} and $A \cap \{ t < \hat{T}\}$, third that $(A \cap \{ t < \hat{T}\})^c = A^c \cup (A \cap \{ \hat{T} \leq t\})$, next the induction hypothesis and $X^1_t/X^2_t = Z^1_t /Z^2_t$ on $\{t <\hat{T}\}$ by $\eqref{eq 4}$, the third step again, and finally that $Z^i_t = Z^i_{\hat{T}}$ on $\{ \hat{T}\leq t \}$ for $i=1,2$ due to Claim 3, we obtain
	\vspace{-0.06cm}
	\begin{align*}
		\liminf_{s \upuparrows \hat{T}} \frac{X_s}{\hat{X}_s}
		& = \liminf_{s \upuparrows \hat{T}} \frac{\Ind_{A^c} X^1_s + \Ind_{A} \big(\frac{X^1_{t \wedge s}}{X^2_t}\big) X^2_{t \vee s}}{\hat{X}_s}
		\\
		& = \Ind_{A^c}\liminf_{s \upuparrows \hat{T}} \frac{ X^1_s}{\hat{X}_s} 
		+ \Ind_{A \cap \{\hat{T} \leq t\}}\liminf_{s \upuparrows \hat{T}} \frac{\big(\frac{X^1_{s}}{X^2_t}\big) X^2_{t }}{\hat{X}_s} 
		+ \Ind_{A \cap \{ t < \hat{T}\}} \liminf_{s \upuparrows \hat{T}} \frac{ \big(\frac{X^1_{t}}{X^2_t}\big) X^2_{s}}{\hat{X}_s}
		\\
		& = \Ind_{(A \cap \{ t < \hat{T}\})^c}\liminf_{s \upuparrows \hat{T}} \frac{ X^1_s}{\hat{X}_s} 
		+ \Ind_{A \cap \{ t < \hat{T}\}} \liminf_{s \upuparrows \hat{T}} \frac{ \big(\frac{X^1_{t}}{X^2_t}\big) X^2_{s}}{\hat{X}_s}
		\\
		& \geq \Ind_{(A \cap \{ t < \hat{T}\})^c} Z^1_{ \hat{T}} + \Ind_{A \cap \{ t < \hat{T}\}}({Z^1_{t}}/{Z^2_t})Z^2_{ \hat{T}}
		\\
		& = \Ind_{A^c} Z^1_{ \hat{T}} + \Ind_{A \cap \{ \hat{T} \leq t\}} (Z^1_{ \hat{T}}/Z^2_{ \hat{T}})Z^2_{ \hat{T}} + \Ind_{A \cap \{ t < \hat{T}\}}({Z^1_{t}}/{Z^2_t})Z^2_{ \hat{T}}
		= Z_{\hat{T}} \mspace{20mu} \Pas
	\end{align*}
	\vspace{-0.06cm}
	This proves Claim 4.
	
	\textit{Claim 5: $\mathcal{X}$ satisfies DSV for $\hat{X}$ if and only if $\mathcal{Z}_\infty$ is bounded in $L^0$.} The \lq\lq if\rq\rq{} direction follows from the observation that $\{\liminf_{s \upuparrows \hat{T}} {X_s}/{\hat{X}_s} : X \in \mathcal{X}\} \subseteq \mathcal{Z}_\infty$. The \lq\lq only if\rq\rq{} direction is a direct consequence of Claim 4, \eqref{eq 7} and Claim 3.
	
	\lq\lq (1) $\Rightarrow$ (2)\rq\rq: Suppose that $\mathcal{X}$ satisfies DSV for $\hat{X}$. Then $\mathcal{Z}$ is an SPP by Claim 2, and by Claim 5, $\mathcal{Z}_\infty$ is bounded in $L^0$. Thus we can apply \cite[Theorem~2.3 and Remark~2.4.4]{Kar13d} (cf. Proposition~\ref{thm kar}) to $\mathcal{Z}$ and on the closed time interval $[0,\infty]$ to obtain a process $Y'$ which is adapted, RCLL in probability, strictly positive on $[0,\infty]$ and such that for each $Z \in \mathcal{Z}$, $ZY'$ is a supermartingale. Since $1 \in \mathcal{Z}$, $Y'$ is itself a supermartingale and hence without loss of generality RCLL \Pas{} (cf. Remark~\ref{rmk 1}). Define $Y := ({Y'}/{\hat{X} })\Ind_{\llbracket 0, \hat{T} \llbracket}$. Then since $\hat{T}$ is a stopping time and $Y'$, $\hat{X}$ are adapted RC \Pas{} and strictly positive on $\llbracket 0, \hat{T} \llbracket$, so is $Y$. Moreover, since $Y'$ is positive on $[0,\infty]$, the minimum principle for supermartingales implies that $\liminf_{t \upuparrows \hat{T}} \hat{X}_t Y_t = \lim_{t \upuparrows \hat{T}}  Y_t' >0$ \Pas{} 
	
	It only remains to show that $Y$ is a supermartingale deflator, i.e.~that for any $X~\in~\mathcal{X}$, \mbox{$A \in \mathcal{F}_s$} and $0 \leq s < t < \infty$, we have $E[X_tY_t \Ind_A] \leq E[X_sY_s \Ind_A]$. On $\{ \hat{T}\leq s \}$, we have \mbox{$\hat{T}<\infty$}, hence $Y_t = Y_s = 0$ by definition, and so \mbox{$E[X_tY_t \Ind_{A \cap \{ \hat{T}\leq s \}}] \leq E[X_sY_s \Ind_{A \cap \{ \hat{T}\leq s \}}]$} is trivially satisfied. On $\{s <\hat{T}\}$, setting \mbox{$Z := (X/\hat{X}) \Ind_{\llbracket 0, \hat{T} \llbracket} + (\liminf_{s \upuparrows \hat{T}} { X_s}/{\hat{X}_s}) \Ind_{\llbracket \hat{T}, \infty \rrbracket}$} and extending $Z(\omega)$ from $[0,\infty)$ to $[0,\infty]$ by keeping it constant on $[\hat{T}(\omega), \infty]$ yields $Z \in \mathcal{Z}$ and we get 
	\begin{align*}
		E[X_tY_t \Ind_{A \cap \{s < \hat{T}\}}]
		& = E\left[0 \Ind_{A \cap \{s < \hat{T} \leq t\}} + \frac{X_t Y'_t}{\hat{X}_t } \Ind_{A \cap \{t < \hat{T}\}}\right]
		= E[Z_t Y'_t \Ind_{A \cap \{t < \hat{T}\}}] 
		\\
		& \leq E[Z_t Y'_t \Ind_{A \cap \{s < \hat{T}\}}] 
		\leq E[Z_s Y'_s \Ind_{A \cap \{s < \hat{T}\}}]
		= E[X_sY_s \Ind_{A \cap \{s < \hat{T}\}}],
	\end{align*}
	where the first equality is due to $X_t=Y_t=0$ on $\{\hat{T} \leq t\}$ and $Y_t = {Y'_t}/{\hat{X}_t }$ on $\{t<\hat{T}\}$, the second holds because $Z_t = X_t/\hat{X}_t$ on $\{t < \hat{T}\}$ by \eqref{eq 4}, the first inequality is due to $Z, Y' \geq 0$ and $s\leq t$, the second holds because $ZY'$ is a supermartingale and $A \cap \{s < \hat{T}\} \in \mathcal{F}_s$ since $\hat{T}$ is a stopping time, and the last equality again follows from the definitions of $Z$ and $Y$. Hence we conclude.
	
	\lq\lq (2) $\Rightarrow$ (1)\rq\rq: Suppose that $(2)$ holds and let $Y$ be an adapted RC supermartingale deflator which is strictly positive on $\llbracket 0, \hat{T} \llbracket$ and satisfies \mbox{$\liminf_{t \upuparrows \hat{T}} \hat{X}_t Y_t >0 $} \Pas{} From $\hat{X} \in \mathcal{X}$ follows that $\hat{X}Y$ is a (nonnegative) supermartingale and hence \mbox{$\liminf_{t \upuparrows \hat{T}} \hat{X}_t Y_t <\infty$} \Pas{} Therefore we obtain that \mbox{$\{\liminf_{t \upuparrows \hat{T}} {X_t}/{\hat{X}_t}: X \in \mathcal{X} \}$} is bounded in $L^0$ (i.e. DSV holds), if and only if 
	$$
	\{(\liminf_{t \upuparrows \hat{T}} {X_t}/{\hat{X}_t}) (\liminf_{t \upuparrows \hat{T}} \hat{X}_t Y_t) : X \in \mathcal{X} \}
	$$ 
	is bounded in $L^0$. Now since $X, \hat{X}$ and $Y$ are all nonnegative, Lemma~\ref{lemma liminf} (c) together with \mbox{$x(s) := X_s (\omega) /\hat{X}_s(\omega)$} and \mbox{$y(s) := \hat{X}_s(\omega) Y_s(\omega)$} (for any fixed $\omega \in \Omega$) implies that
	$$
	(\liminf_{t \upuparrows \hat{T}} {X_t}/{\hat{X}_t}) (\liminf_{t \upuparrows \hat{T}} \hat{X}_t Y_t) 
	\leq \liminf_{t \upuparrows \hat{T}} X_t Y_t,
	$$
	and so it suffices to show that $\left\{\liminf_{t \upuparrows \hat{T}} X_t Y_t: X \in \mathcal{X} \right\}$ is bounded in $L^0$. If this is not the case, there is a $\delta>0$ and for each $N\in \mathbb{N}$ some $X^N \in \mathcal{X}$ with \mbox{$P[\liminf_{t \upuparrows \hat{T}} X^N_t Y_t >N]>\delta$}. Fix $N\in\mathbb{N}$, define \mbox{$\tau^N := \inf \{ t\geq 0: X^N_t Y_t \geq N\} \wedge \hat{T}$} and note that $\tau^N$ is a stopping time. Then \mbox{$P[\tau^N < \hat{T}]>\delta$} and the right-continuity of $X^NY$ gives $X^N_{\tau^N}Y_{\tau^N}\geq N$ on $\{ \tau^N < \hat{T} \}$. Since $X^N$ and $Y$ are nonnegative, it follows that $E[X^N_{\tau^N}Y_{\tau^N}]>\delta N$. On the other hand, $X^NY$ is a nonnegative supermartingale, hence right-closable at $\infty$ and so the optional stopping theorem on $[0,\infty]$ yields $E[X^N_{\tau^N}Y_{\tau^N}] \leq E[X^N_0Y_0] = 1$. This is a contradiction for any $N>1/\delta$, and hence we conclude.
\end{proof}

%%%%%%%%%%%%%%%%%%%%%%%%%%%%%%%%%%%%%%%%%%%%%%%%%%%%%%%%%%%%%%%%%%%%%%%%%%%%%%%%%%%%%%%%%%%%%%%%

\subsection{Proof of Proposition~\ref{prop 1} and Corollary~\ref{cor 2}}

In this section we give the proof of Proposition~\ref{prop 1}. Using Proposition~\ref{prop 1} we are then able to quickly prove Corollary~\ref{cor 2}.

\begin{proof}[Proof of Proposition~\ref{prop 1}]	
	Define $\tau^X := \inf \{ t\geq 0: X_t =0\}$. First we show that 
	\begin{equation}\label{eq 9}
		X = X\Ind_{\llbracket 0, \tau^X \llbracket} \mspace{20mu} \Pas, \mspace{20mu} \forall X \in \mathcal{X}.
	\end{equation}
	Under (a), this is trivial. If (b) holds, suppose to the contrary that there exist an \mbox{$X \in \mathcal{X}$} and \mbox{$t \in [0,\infty)$} such that $X_t \Ind_{\{ \tau^X \leq t\}}>0$ with positive probability. Since $Y$ is an SMD for $\mathcal{X}$, $E[X_t Y_t | \mathcal{F}_{t \wedge \tau^X}] \leq X_{t \wedge \tau^X} Y_{t \wedge \tau^X}$ \Pas{} and hence \mbox{$E[X_t Y_t \Ind_{\{ \tau^X \leq t\}} | \mathcal{F}_{t \wedge \tau^X}] \leq X_{t \wedge \tau^X} Y_{t \wedge \tau^X} \Ind_{\{ \tau^X \leq t\}}$} \Pas{} Since $X$ is \Pas{} RC, $X_{\tau^X} = 0$ \Pas{} and hence $\{\tau^X \leq t\} = \{t \wedge \tau^X = \tau^X \}$ implies \mbox{$X_{t \wedge \tau^X} Y_{t \wedge \tau^X}\Ind_{\{ \tau^X \leq t\}} = 0$} \Pas{} On the other hand, $X_tY_t \Ind_{\{ \tau^X \leq t\}} >0$ with positive probability due to \mbox{$\{Y=0\}\subseteq \{X=0\}$}, and hence $E[X_t Y_t \Ind_{\{\tau^X\leq t\}}| \mathcal{F}_{t \wedge \tau^X}]  > 0$ with positive probability, which is a contradiction. This proves~\eqref{eq 9}.
	
	Define now the set of cemetery times as $\mathcal{T} := \{  \tau^X: X \in \mathcal{X}\}$. Then each $\tau \in \mathcal{T}$ is a stopping time. Define furthermore the function $g : \mathcal{T} \rightarrow \mathbb{R}$ by \mbox{$g(\tau) := E[\int_0^\tau e^{-t} dt]$}. Then \mbox{$0 \leq M := \sup_{\tau  \in \mathcal{T}} g(\tau) \leq E[\int_0^\infty e^{-t} dt] = 1$}. Let $(\tau^n)_{n \in \mathbb{N}}$ be a sequence which satisfies \mbox{$g(\tau^n) \geq M - 1/n$} for each $n\in \mathbb{N}$ and define $\hat{T}^n := \max_{k = 1, \dots, n} \tau^k$. Then $\hat{T}^n$ is an element of $\mathcal{T}$ because $\mathcal{T}$ is closed under taking maxima; indeed, for $X^1, X^2 \in \mathcal{X}$, we have \mbox{$\tau^{X^1} \vee \tau^{X^2} = \tau^{cc^{1/2}(X^1, X^2)} \in \mathcal{T}$} due to \eqref{c}. So $(\hat{T}^n)_{n \in \mathbb{N}}$ is an increasing sequence in $\mathcal{T}$ with $g(\hat{T}^n) \geq M - 1/n$ for each $n\in\mathbb{N}$. Hence $\widetilde{T} := \lim_{n \rightarrow \infty} \hat{T}^n = \sup_{n \in \mathbb{N}} \hat{T}^n$ exists and is a $[0,\infty]$-valued stopping time. Moreover, $g(\widetilde{T})=M$ since $g(\widetilde{T})\geq g(\hat{T}^n)$ for all $n\in \mathbb{N}$. We claim that for each $X \in \mathcal{X}$, we have $\tau^X \leq \widetilde{T}$ \Pas{} To see this, suppose to the contrary that there is an $\bar{X} \in \mathcal{X}$ with $\tau^{\bar{X}} > \widetilde{T}$ with positive probability. The sequence $(\bar{\tau}^n)_{n\in \mathbb{N}}$ defined by $\bar{\tau}^n := \tau^{\bar{X}} \vee \hat{T}^n \in \mathcal{T}$ converges to $\tau^{\bar{X}} \vee \widetilde{T}$ \Pas{} Note that $g(\widetilde{T} \vee \tau^{\bar{X}})>g(\widetilde{T})$ because $P[\tau^{\bar{X}}>\widetilde{T}]>0$, and thus by monotone convergence, there is an $n\in\mathbb{N}$ with $g(\bar{\tau}^n)>g(\widetilde{T})$, which is a contradiction. But now for each $X \in \mathcal{X}$, due to \eqref{eq 9}, $\tau^X \leq \widetilde{T}$ \Pas{} if and only if $X = X\Ind_{\llbracket 0, \widetilde{T} \llbracket}$ \Pas{} Hence $\widetilde{T}$ is \Pas{} unique and we conclude.
\end{proof}

\begin{proof}[Proof of Corollary~\ref{cor 2}]
	Since \lq\lq (2) $\Rightarrow$ (1)\rq\rq{} follows immediately from Proposition~\ref{prop 1} and Theorem~\ref{thm 1}, we only need to argue \lq\lq (1) $\Rightarrow$ (2)\rq\rq. Assume that (1) holds; then Theorem~\ref{thm 1} yields an SMD $\thinspace Y'$ which is adapted, \Pas{} RC, strictly positive on $\llbracket 0, \hat{T} \llbracket$ and satisfies $\liminf_{t \upuparrows \hat{T}} \hat{X}_t Y_t' >0$ \Pas{} Define $Y := Y'\Ind_{\llbracket 0, \hat{T} \llbracket} + \Ind_{\llbracket \hat{T}, \infty \rrbracket}$. Then $Y$ is clearly strictly positive, \Pas{} RC, and since $\hat{T}$ is a stopping time, $Y$ is adapted. Since $0$ is an absorbing state, $X=0$ on $\llbracket \hat{T},\infty \rrbracket$ for each $X \in \mathcal{X}$. But then $X_tY_t = X_tY_t'$ for each $t \in [0,\infty)$ and $X \in \mathcal{X}$ and hence $Y$ is an SMD. Finally, $\liminf_{t \upuparrows \hat{T}} \hat{X}_t Y'_t \geq \liminf_{t \upuparrows \hat{T}} \hat{X}_t Y_t >0$ \Pas{} due to $Y'>Y$; this finishes the proof.
\end{proof}

%%%%%%%%%%%%%%%%%%%%%%%%%%%%%%%%%%%%%%%%%%%%%%%%%%%%%%%%%%%%%%%%%%%%%%%%%%%%%%%%%%%%%%%%%%%%%%%%

\subsection{Proof of Theorem~\ref{thm 2}}

In order to prove Theorem~\ref{thm 2}, we need some preliminary work. The next lemma shows that if we enlarge our filtration by a finite family $\mathcal{V}$ of sets and extend our initial SP $\mathcal{X}$ to $\mathcal{X}^\mathcal{V}$ in a minimal way such that $\mathcal{X}^\mathcal{V}$ is an SP over the enlarged filtration, then this $\mathcal{X}^\mathcal{V}$ satisfies NUPBR$_{\rm loc}$ whenever $\mathcal{X}$ does.

\begin{lemma}\label{lemma 8}
	Let $\mathcal{X}$ be an SP, $\mathcal{V}$ a finite family of subsets of $\Omega$ and $\mathbb{F}^\mathcal{V} := (\mathcal{F}^\mathcal{V}_t )_{t \geq 0}$ with $\mathcal{F}^\mathcal{V}_t := \sigma(\mathcal{F}_t \cup \mathcal{V}) = \mathcal{F}_t \vee \sigma(\mathcal{V})$. Let $\mathcal{X}^\mathcal{V}$ be the smallest set containing $\mathcal{X}$ such that \eqref{c} and \eqref{d} are satisfied over $(\Omega, \mathcal{F}^\mathcal{V}_\infty, \mathbb{F}^\mathcal{V}, P)$. Then $\mathcal{X}^\mathcal{V}$ is an SP. If $\mathcal{X}$ satisfies NUPBR$_{\rm loc}$, then so does $\mathcal{X}^\mathcal{V}$.
\end{lemma}

\begin{proof}
	Let $\mathcal{V}' = \{ V^1, \dots, V^K\}$ for some $K \in \mathbb{N}$ be a finite partition of $\Omega$ into pairwise disjoint (not necessarily measurable) sets such that $\sigma(\mathcal{V}') = \sigma(\mathcal{V})$. Define $\mathcal{X}^0 := \mathcal{X}$ and recursively
	\begin{align*}
		\mathcal{X}^{n}
		:= & \left\{cc^\alpha (X^1, X^2) : X^1, X^2 \in \mathcal{X}^{n-1}, \alpha \in[0,1] \right\} 
		\\
		& \cup \left\{sw^{t, A} (X^1, X^2) : X^1, X^2 \in \mathcal{X}^{n-1} , t \in [0,\infty), A \in \mathcal{F}^\mathcal{V}_t \right\} \mspace{20mu} \text{for $n \in \mathbb{N}$}.
	\end{align*}
	Our first claim is that 
	\begin{equation}\label{eq 11}
		\mathcal{X}^\mathcal{V} = \mathcal{X}' := \bigcup_{n \in \mathbb{N}} \mathcal{X}^n.
	\end{equation}
	Because $\mathcal{X}'$ is by construction closed under $cc^{\alpha}$ with $\alpha \in [0,1]$ and $sw^{t,A}$ with $t\in [0,\infty)$ and $A \in \mathcal{F}^\mathcal{V}_t$, the inclusion \lq\lq $\subseteq$\rq\rq{} is clear. The converse \lq\lq $\supseteq$\rq\rq{} follows by induction from the fact that $\mathcal{X}^n\subseteq \mathcal{X}^\mathcal{V}$ for each $n\in\mathbb{N}$. Since $\mathcal{X}^n$ satisfies \eqref{a} for each $n \in \mathbb{N}$, so does $\mathcal{X}^\mathcal{V}$; this proves \eqref{eq 11}.
	
	Next we claim that 
	\begin{equation}\label{eq 10}
		\text{for each $X \in \mathcal{X}'$, there are $X^k \in \mathcal{X}$ for $k=1, \dots, K$ with $X=X^k$ on $V^k$}.	
	\end{equation}
	To show this, we use \eqref{eq 11} and induction over $n$. For $n=0$ and $X \in \mathcal{X}^0$, the conclusion in \eqref{eq 10} holds trivially. For the induction step, let $X \in \mathcal{X}^n$ and first consider the case when \mbox{$X=cc^\alpha (X^1, X^2)$} for some $X^1, X^2 \in \mathcal{X}^{n-1}$, $\alpha \in[0,1]$. By the induction hypothesis, there are $X^{k, i} \in \mathcal{X}$ for $k=1,\dots, K$ and $i = 1,2$ such that $X^i=X^{k, i}$ on $V^k$. Then for each $k = 1, \dots, K$, $X^k := cc^\alpha (X^{k,1}, X^{k,2})$ is in $\mathcal{X}$ and satisfies $X=X^k$ on $V^k$. Consider now the case when $X=sw^{t,A}(X^1, X^2)$ for $X^1, X^2 \in \mathcal{X}^{n-1}$, $t \in [0,\infty)$, $A \in \mathcal{F}^\mathcal{V}_t$. By the induction hypothesis again, there are $X^{k, i} \in \mathcal{X}$ for $k=1,\dots, K$ and $i = 1,2$ such that $X^i=X^{k, i}$ on $V^k$. Note that by Lemma~\ref{lemma 7} below, $A$ can be written as $A = \bigcup_{k = 1}^K (V^k \cap A^k)$ for some $A^k \in \mathcal{F}_t$, $k =1, \dots, K$. Then for each $k = 1, \dots, K$, $X^k := sw^{t,A^k}(X^{k,1}, X^{k,2})$ is in $\mathcal{X}$ and satisfies $X=X^k$ on $V^k$; this proves \eqref{eq 10}. 
	
	Suppose now that $\mathcal{X}^\mathcal{V}$ does not satisfy NUPBR$_{\rm loc}$. Then there exist $t \in [0, \infty)$ and $\delta>0$ such that for each $n \in \mathbb{N}$, there is an $X^n \in \mathcal{X}^\mathcal{V}$ with \mbox{$P[ X^n_t > n ]>\delta$}. Fix $n \in \mathbb{N}$. Then there exists $k(n)$ such that \mbox{$P[X^n_t \Ind_{\mathcal{V}^{k(n)}} > n]>\delta/K$}, due to the pigeonhole principle. But by \eqref{eq 10} and \eqref{eq 11}, there exists $X^{n, k(n)} \in \mathcal{X}$ with $X^n \Ind_{\mathcal{V}^{k(n)}} = X^{n,k(n)} \Ind_{\mathcal{V}^{k(n)}}$ and hence \mbox{$P[X^{n, k(n)}_t > n]>\delta/K$}. It follows that $\mathcal{X}_t$ is not bounded in $L^0$, and hence $\mathcal{X}$ does not satisfy NUPBR$_{\rm loc}$ either. This finishes the proof.
\end{proof}

The next lemma shows that even though NUPBR$_{\rm loc}$ requires boundedness in $L^0$ of $\mathcal{X}_t$ only for deterministic times $t \in [0,\infty)$, it implies boundedness in $L^0$ of $X_\tau$ also for random times $\tau$ which take at most countably many values.

\begin{lemma}\label{lemma 3}
	Let $\mathcal{X}$ be an SP which satisfies NUPBR$_{\rm loc}$. Then $\mathcal{X}_\tau$ is bounded in $L^0$ for any $\tau: \Omega \rightarrow [0,\infty)$ which takes at most countably many values.
\end{lemma}

\begin{proof}
	Take such a $\tau$ and suppose that there exists $\delta>0$ such that for each $n \in \mathbb{N}$, there is an $X^n\in\mathcal{X}$ with $P[ X^n_\tau > n ]>\delta$. Set $A^n := \{X^n_\tau > n \}$ so that $P[A^n]>\delta$ for each $n \in \mathbb{N}$ and let $(C^k)_{k \in \mathbb{N}}$ be a countable partition of $\Omega$ into pairwise disjoint measurable sets such that $\tau$ is constant on $C^k$ for each $k$, with value $\tau(C^k)$. Define \mbox{$f(k, N) := \sup \{ P[C^k \cap A^n] : n\geq N\}$}; then $f$ is of course decreasing in $N$. Suppose that there are $c >0$ and $k \in \mathbb{N}$ such that $f(k,N)>c$ for each $N \in \mathbb{N}$. Then for each $N\in \mathbb{N}$, there exists $n>N$ such that \mbox{$P[X^n_{\tau(C^k)} > n] \geq P[C^k \cap A^n]> f(k, N)/2>c/2$} and hence $\mathcal{X}_{\tau(C^k)}$ is not bounded in $L^0$. In particular, NUPBR$_{\rm loc}$ does not hold, which is a contradiction. It follows that $\lim_{N \rightarrow \infty}f(k,N) = 0$ for each $k \in \mathbb{N}$. Pick $K \in \mathbb{N}$ such that $\sum_{k > K}P[C^k]<\delta/2$ and choose $N \in \mathbb{N}$ such that $f(k,N)<\delta/(2K)$ for each $k \leq K$. Then 
	$$
	P[A^N] = \sum_{k > K} P[C^k \cap A^N] + \sum_{k \leq K} P[C^k \cap A^N] < \delta/2 + \sum_{k \leq K} f(k, N) < \delta,
	$$ 
	which is a contradiction. This finishes the proof.
\end{proof}

Before we start proving Theorem~\ref{thm 2}, we give a brief outline of the proof of \lq\lq (1)~$\Rightarrow$~(2)\rq\rq, which is the non-standard direction. We start by creating a localising sequence $(\tau^n)_{n \in\mathbb{N}}$ for $\widetilde{T}$ as in Proposition~\ref{prop 1} such that each $\tau^n$ takes only finitely many values and there is a strictly positive process on $\llbracket 0, \tau^n \rrbracket$. This we achieve by discretising the sequence $\hat{T}^n$ given by Proposition~\ref{prop 1}. The obtained sequence $(\tau^n)_{n \in\mathbb{N}}$ is not measurable; so we enlarge the filtration such that each $\tau^n$ becomes a stopping time. Then we define an auxiliary set $\mathcal{Z}^{+,n}$ for each $\tau^n$ similarly as in the proof of Theorem~\ref{thm 1}. Lemmas~\ref{lemma 8} and \ref{lemma 3} ensure that $\mathcal{Z}^{+, n}_\infty$ is bounded in $L^0$ whenever $\mathcal{X}$ satisfies NUPBR$_{\rm loc}$, and hence an application of Kardaras \mbox{\cite[Theorem~2.3]{Kar13d}} to each $\mathcal{Z}^{+,n}$ yields for each $n\in\mathbb{N}$ an SMD $Y^{+,n}$ for $\mathcal{Z}^{+,n}$ over the enlarged filtration (step~1). We paste $(Y^{+,i})_{i=1, \dots, n}$ together, yielding $\bar{Y}^{+,n}$, and argue that $\bar{Y}^{+,n}$ is still an SMD for $\mathcal{Z}^{+,n}$ over the enlarged filtration (step~2). Taking first an adequate truncation of $\bar{Y}^{+,n}$ (step~3) and then the optional projection onto the original filtration yields a process $Y^n$ which is strictly positive on $\llbracket 0, \tau^n \rrbracket$ and an SMD for $\mathcal{X}$ over the original filtration (step~4). Finally, we observe that $(Y^n)_{n\in\mathbb{N}}$ converges monotonically to some $Y$ and verify that $Y$ is strictly positive on $\llbracket 0, \widetilde{T} \llbracket$ and an SMD for $\mathcal{X}$ (step~5).

\begin{proof}[Proof of Theorem~\ref{thm 2}]
	\lq\lq (1) $\Rightarrow$ (2)\rq\rq: \textit{Step 1 (enlargement):} Since $0$ is an absorbing state, we can apply Proposition~\ref{prop 1} via Remark~\ref{rmk 2} to obtain $\widetilde{T}$ and $(\hat{X}^n)_{n \in \mathbb{N}} \subseteq \mathcal{X}$ with \mbox{$X = X\Ind_{\llbracket 0, \widetilde{T} \llbracket}$} (surely) for all $X \in \mathcal{X}$ and such that the sequence \mbox{$\hat{T}^n := \inf\{t\geq0: \hat{X}^n=0\}$} satisfies $\hat{T}^n \nearrow \widetilde{T}$ as $n \rightarrow \infty$. For $n \in \mathbb{N}$ define \mbox{$C^{n,k} :=\{k/2^n < \hat{T}^n \leq (k+1)/2^n \}$} for \mbox{$k=0, \dots, n2^n-1$}, and $C^{n,n2^n} :=\{n < \hat{T}^n \}$, which gives a finite partition of $\Omega$ into pairwise disjoint sets; note that $\hat{T}^n>0$ because $\hat{X}^n$ is RC and starts at $1$. Note for later use that for any $n \in \mathbb{N}$ and $K< n2^n$, due to $\hat{T}^n \leq \hat{T}^{n+1}$, we have
	\begin{equation} \label{eq 6}
		\bigcup_{k = 0}^K C^{n,k} 
		= \left\{ \hat{T}^n \leq \frac{K+1}{2^n} \right\} 
		\supseteq \left\{ \hat{T}^{n+1} \leq \frac{2K+2}{2^{n+1}} \right\} 
		= \bigcup_{k = 0}^{2K+1} C^{n+1,k},
	\end{equation}
	whereas for $K= n2^n$, we have $\bigcup_{k = 0}^K C^{n,k} = \Omega \supseteq \bigcup_{k = 0}^{2K+1} C^{n+1,k}$. For each $n \in \mathbb{N}$, define now $\tau^n := \sum_{k=0}^{n2^n} (k/2^n) \Ind_{C^{n,k}}$; then $\tau^n <\hat{T}^n \leq \widetilde{T}$ \Pas{} and $\lim_{n\rightarrow \infty} \tau^n = \widetilde{T}$.
	Fix $n$ and for each $t \in [0,\infty)$, let $K(t, n)$ be the smallest integer such that $t < 2K(t, n)/2^n$ if $t < n$ and $2K(t, n)=n2^n$ otherwise. Define recursively $\mathcal{F}^{+, 0}_t := \mathcal{F}_t$ and for $n \geq 1$,
	\begin{equation}\label{eq 8}
		\mathcal{F}^{+,n}_t := \sigma(\mathcal{F}^{+, n-1}_t \cup \{C^{n,k}: k= 0,\dots, 2K(t,n)-1 \})
	\end{equation}
	and let $\mathbb{F}^{+, n} = (\mathcal{F}^{+, n}_t)_{t \geq 0}$. Then $\mathbb{F}^{+, n}$ satisfies like $\mathbb{F}$ the usual conditions and $\tau^n$ is an $\mathbb{F}^{+,n}$-stopping time. Indeed, for $t\geq n$, $\{\tau^n \leq t\} = \Omega$ and for any $t < n$, we have that either $(2K(t, n)-2)/2^n \leq t < (2K(t, n)-1)/2^n$ or $(2K(t, n)-1)/2^n \leq t < 2K(t, n)/2^n$. By noting that $K(t,n+1)-1 = 2K(t,n)-2$ in the first case and $K(t,n+1)-1 = 2K(t,n)-1$ in the second, we get for $t < n$ that
	\begin{equation}\label{eq 5}
		\{\tau^n \leq t\} 
		= \left\{ \sum_{k=0}^{n 2^n} (k/2^n) \Ind_{C^{n,k}} \leq t \right\}
		= \bigcup_{k = 0}^{K(t,n+1)-1} C^{n,k} \in \mathcal{F}^{+ , n}_t,
	\end{equation}
	Let $\mathcal{X}^{+, n}$ be the smallest set containing $\mathcal{X}$ and satisfying \eqref{c} and \eqref{d} over $(\Omega, \mathcal{F}^{+, n}_\infty, \mathbb{F}^{+, n}, P)$. Then $\mathcal{X}^{+, n}$ is an SP by Lemma~\ref{lemma 8}. Moreover, since $\mathbb{F}^{+,n}$ is coarser than $\mathbb{F}^{\mathcal{V}^n}$ where \mbox{$\mathcal{V}^n := \{C^{i, k}: i = 1, \dots, n, k= 0,\dots, i2^i \}$} by \eqref{eq 8}, it follows that $\mathcal{X}^{+, n}$ is a subset of $\mathcal{X}^{\mathcal{V}^n}$ as defined in Lemma~\ref{lemma 8}, and since the latter satisfies NUPBR$_{\rm loc}$ by Lemma~\ref{lemma 8}, so does $\mathcal{X}^{+, n}$.
	
	For fixed $n$, consider the set $\mathcal{Z}^{+, n} : = \{ Z^n := X\Ind_{\llbracket 0, \tau^n \llbracket} + X_{\tau^n}\Ind_{\llbracket \tau^n, \infty \rrbracket} : X \in \mathcal{X}^{+, n} \}$ and extend each $Z_\cdot^n (\omega)$ from $[0,\infty)$ to $[0,\infty]$ by keeping it constant on $[\tau^n(\omega), \infty]$. According to Remark~\ref{rmk 3}, $\mathcal{Z}^{+, n}$ is a set of processes with bounded time horizon (since it lives on the right-closed interval $[0,\infty]$). In particular, the conditions \eqref{a}--\eqref{d} and \eqref{B} are to be understood on the right-closed interval $[0,\infty]$. Since $\tau^n$ is an $\mathbb{F}^{+, n}$-stopping time, $Z^{X,n}:=X_{\tau^n \wedge \cdot}$ is $\mathbb{F}^{+, n}$-adapted for each $X \in \mathcal{X}^{+,n}$. Moreover, each $Z^{X,n}$ is clearly nonnegative, RC and satisfies $Z^{X,n}_0=X_0 = 1$, and so $\mathcal{Z}^{+, n}$ satisfies \eqref{a}. Define $\hat{Z}^n := \hat{X}^n_{\tau^n \wedge \cdot} \in \mathcal{Z}^{+, n}$ and recall that \mbox{$\hat{X}^n \in \mathcal{X}$}. Then $\hat{Z}^n$ is strictly positive on $[0, \infty)$ due to $\tau^n < \hat{T}^n$, and hence $\mathcal{Z}^{+, n}$ satisfies \eqref{b}. It is straightforward to verify that $\mathcal{Z}^{+, n}$ satisfies \eqref{c} and \eqref{d} from these properties of $\mathcal{X}^{+, n}$, and hence $\mathcal{Z}^{+, n}$ is an SPP. Furthermore, since $\tau^n$ takes finitely many values, Lemma~\ref{lemma 3} implies that $\mathcal{Z}^{+, n}_\infty = \mathcal{X}^{+, n}_{\tau^n}$ is bounded in $L^0$. Hence we can apply \cite[Theorem~2.3 and Remark~2.4.4]{Kar13d} (cf. Proposition~\ref{thm kar}) to $\mathcal{Z}^{+, n}$ over $(\Omega, \mathcal{F}^{+, n}_\infty, \mathbb{F}^\mathcal{V}, P)$ and on the closed time interval $[0,\infty]$ to obtain a process $Y^{+, n}$ which is $\mathbb{F}^{+, n}$-adapted,
	%	RCLL in probability,
	strictly positive on $[0,\infty]$ and such that $ZY^{+, n}$ is an $\mathbb{F}^{+, n}$-supermartingale for each $Z \in \mathcal{Z}^{+, n}$. We can also without loss of generality assume that $Y^{+, n}_0 = 1$. 
	
	%	RC
	%	Observe that $\hat{Z}^n Y^{+, n}$ is an $\mathbb{F}^{+, n}$-supermartingale which is RC in probability, and hence without loss of generality RCLL \Pas{} (cf. Remark~\ref{rmk 1}). Since $\hat{Z}^n$ is \Pas{} RC and strictly positive, $Y^{+, n} = (\hat{Z}^n Y^{+, n})/\hat{Z}^n$ is also \Pas{} RC.
	
	\textit{Step 2 (pasting):} Set $\tau^0 = 0$ and define $\bar{Y}^{+, n} := \prod_{i = 1}^{n} (Y^{+, i}_{\tau^i \wedge \cdot}/Y^{+, i}_{\tau^{i-1} \wedge \cdot})$ for $n \in \mathbb{N}$. Then each $\bar{Y}^{+, n}$ is a finite product of $\mathbb{F}^{+, n}$-adapted
	%	, \Pas{} RC 
	and strictly positive processes, hence $\mathbb{F}^{+, n}$-adapted
	%	, \Pas{} RC 
	and strictly positive itself, and moreover $\bar{Y}^{+, n}_0 = 1$. We show by induction that for every $n \in \mathbb{N}$, $\bar{Y}^{+, n}$ is an $\mathbb{F}^{+, n}$-supermartingale deflator for $\mathcal{Z}^{+, n}$. For $n=1$, this is clear because $\bar{Y}^{+, 1} = Y^{+, 1}$ is an $\mathbb{F}^{+, 1}$-supermartingale deflator for $\mathcal{Z}^{+, 1}$. For the induction step, note that 
	\begin{equation}\label{eq 1}
		\bar{Y}^{+, n+1} = \bar{Y}^{+, n} \text{ on } \llbracket 0, \tau^n \rrbracket 
		\mspace{20mu} \text{ and } \mspace{20mu}
		\bar{Y}^{+, n+1} = (\bar{Y}^{+, n}_{\tau^n}/Y^{+, n+1}_{\tau^n})Y^{+, n+1} \text{ on } \llbracket \tau^n, \infty \rrbracket.
	\end{equation}
	Let $0\leq s <t <\infty$, $A \in \mathcal{F}_s^{+, n+1}$ and $Z\in \mathcal{Z}^{+, n+1}$. Consider $A^s := A \cap \{s < \tau^n\}$ and note that $A^s$ is an element of $\mathcal{F}_s^{+, n}$. Indeed, if $s\geq n$, then $\{s <\tau^n\} = \emptyset$ and $A_s = \emptyset \in \mathcal{F}_s^{+,n}$, whereas if $s < n$, then \eqref{eq 5} and \eqref{eq 6} give 
	$$
	\{\tau^n \leq s\} = \bigcup_{k = 0}^{K(s,n+1)-1} C^{n,k} \supseteq \bigcup_{k = 0}^{2K(s,n+1)-1} C^{n+1,k},
	$$ 
	or $\Omega \backslash \bigcup_{k = 0}^{2K(s,n+1)-1} C^{n+1, k} \supseteq \{s<\tau^n\}$. But due to \eqref{eq 8} and Lemma~\ref{lemma 7}, $A$ can be written as $A = \bigcup_{k = 0}^{2K(s,n+1)-1} (C^{n+1, k} \cap A^k) \cup ((\Omega \backslash \bigcup_{k = 0}^{2K(s,n+1)-1} C^{n+1, k}) \cap B)$ for some \mbox{$A^k, B \in \mathcal{F}_s^{+, n}$}, $k=1, \dots, 2K(s, n+1)-1$ and hence $A_s = B \cap \{s < \tau^n\} \in \mathcal{F}_s^{+,n}$. So we have
	\begin{align*}
		E[Z_t \bar{Y}^{+, n+1}_t \Ind_{A^s}] 
		& = E[Z_t \bar{Y}^{+, n+1}_t \Ind_{A^s \cap \{ t \leq \tau^n \}} + Z_t \bar{Y}^{+, n+1}_t \Ind_{A^s \cap \{ t > \tau^n \}}]
		\\
		& = E\bigg[Z_t \bar{Y}^{+, n}_t \Ind_{A^s \cap \{ t \leq \tau^n \}} 
		+ E\bigg[ Z_t Y^{+, n+1}_t \frac{\bar{Y}^{+, n}_{\tau^n}}{Y^{+, n+1}_{\tau^n}} \Ind_{A^s \cap \{ t > \tau^n \}}\bigg|\mathcal{F}_{t \wedge \tau^n}^{+, n+1}\bigg]\bigg]
		\\
		& \leq E\bigg[Z_t \bar{Y}^{+, n}_t \Ind_{A^s \cap \{ t \leq \tau^n \}} 
		+ Z_{\tau^n} Y^{+, n+1}_{\tau^n} \frac{\bar{Y}^{+, n}_{\tau^n}}{Y^{+, n+1}_{\tau^n}} \Ind_{A^s \cap \{ t > \tau^n \}}\bigg]
		\\
		& = E[Z_{t \wedge \tau^n} \bar{Y}^{+, n}_{t \wedge \tau^n} \Ind_{A^s}] 
		\leq E[Z_s \bar{Y}^{+, n}_s \Ind_{A^s}]
		= E[Z_s \bar{Y}^{+, n+1}_s \Ind_{A^s}],
	\end{align*}
	where the second equality is due to \eqref{eq 1} and the tower property; the first inequality follows from the optional sampling theorem since $ZY^{+, n+1}$ is an $\mathbb{F}^{+, n+1}$-supermartingale, $t \wedge \tau^n$ is a bounded stopping time and $(\bar{Y}^{+, n}_{\tau^n}/Y^{+, n+1}_{\tau^n}) \Ind_{A^s \cap \{ t > \tau^n \}}$ is $\mathcal{F}^{+, n+1}_{t \wedge \tau^n}$-measurable; the second follows from the induction hypothesis and the fact that $Z_{\tau^n \wedge \cdot} \in \mathcal{Z}^{+, n}$, $A^s \in \mathcal{F}^{+,n}_s$; and the last equality is due to \eqref{eq 1} again because $s<\tau^n$ on $A^s$. Consider now the set \mbox{$A^\ell := A \cap \{s \geq \tau^n\}  \in \mathcal{F}_s^{+, n+1}$}; then	using twice \eqref{eq 1} and the $\mathbb{F}^{+, n+1}$-supermartingale property of $ZY^{+, n+1}$ plus the fact that $(\bar{Y}^{+, n}_{\tau^n}/Y^{+, n+1}_{\tau^n}) \Ind_{A^\ell}$ is $\mathcal{F}_s^{+, n+1}$-measurable yield
	\begin{align*}
		E[Z_t \bar{Y}^{+, n+1}_t \Ind_{A^\ell}] 
		&= E[Z_t Y^{+, n+1}_t (\bar{Y}^{+, n}_{\tau^n}/Y^{+, n+1}_{\tau^n}) \Ind_{A^\ell}] 
		\\
		& \leq E [Z_s Y^{+, n+1}_s(\bar{Y}^{+, n}_{\tau^n}/Y^{+, n+1}_{\tau^n}) \Ind_{A^\ell}]
		= E[Z_s \bar{Y}^{+, n+1}_s \Ind_{A^\ell}].
	\end{align*}
	Hence $E[Z_t \bar{Y}^{+, n+1}_t \Ind_{A}]  \leq E[Z_s \bar{Y}^{+, n+1}_s \Ind_{A}]$, and this proves the induction step.
	
	\textit{Step 3 (cutoff):} For each $n\in\mathbb{N}$, also $\bar{Y}^{+, n} \Ind_{\llbracket 0, \tau^n \rrbracket}$ is an $\mathbb{F}^{+, n}$-SMD for $\mathcal{Z}^{+, n}$. Indeed, for any $0\leq s <t <\infty$, $A \in \mathcal{F}^{+, n}_s$ and $Z\in \mathcal{Z}^{+, n}$, we have
	\begin{equation*}
		E[Z_t \bar{Y}^{+, n}_t \Ind_{\{t \leq \tau^n\}}  \Ind_{A}] 
		\leq E[Z_t \bar{Y}^{+, n}_t \Ind_{\{ s\leq \tau^n \}} \Ind_{A}] 
		\leq E[Z_s \bar{Y}^{+, n}_s \Ind_{\{ s\leq \tau^n \}} \Ind_{A}],
	\end{equation*}	
	where the first inequality follows from the nonnegativity of $Z$ and $\bar{Y}^{+, n}$, and the second holds because $Z\bar{Y}^{+, n}$ is an $\mathbb{F}^{+, n}$-supermartingale and $A, \{s \leq \tau^n\}$ are in $\mathcal{F}^{+, n}_s$.
	
	\textit{Step 4 (projection):} For each $n\in \mathbb{N}$, let $Y^n$ be the optional projection of the nonnegative process $\bar{Y}^{+, n} \Ind_{\llbracket 0, \tau^n \rrbracket}$ with respect to $\mathbb{F}$, cf.~\cite[VI.43]{DellacherieMeyer82}. Then each $Y^n$ is $\mathbb{F}$-adapted and strictly positive on $\llbracket 0, \tau^n \rrbracket$, and furthermore for each $t \in [0,\infty)$, \mbox{$Y^n_t = E[\bar{Y}^{+, n}_t \Ind_{\llbracket 0, \tau^n \rrbracket}(t)| \mathcal{F}_t]$}. To show that $Y^n$ is an $\mathbb{F}$-SMD for $\mathcal{X}$, fix $X \in \mathcal{X}$ and $s\leq t$. Then by multiple applications of the tower rule, we obtain
	\begin{align*}
		E[X_t Y^n_t | \mathcal{F}_s] 
		& = E\big[X_t E[ \bar{Y}^{+, n}_t \Ind_{\llbracket 0, \tau^n \rrbracket}(t)| \mathcal{F}_t] \big| \mathcal{F}_s\big]
		\\
		& = E[X_t \bar{Y}^{+, n}_t \Ind_{\llbracket 0, \tau^n \rrbracket}(t) | \mathcal{F}_s]
		\\
		& = E\big[E[X_t \bar{Y}^{+, n}_t \Ind_{\llbracket 0, \tau^n \rrbracket}(t) | \mathcal{F}^{+, n}_s] \big| \mathcal{F}_s\big]
		\\
		& = E\big[E[Z^{X,n}_t \bar{Y}^{+, n}_t \Ind_{\llbracket 0, \tau^n \rrbracket}(t) | \mathcal{F}^{+, n}_s] \big| \mathcal{F}_s\big]
		\\
		& \leq E[Z^{X,n}_s \bar{Y}^{+, n}_s \Ind_{\llbracket 0, \tau^n \rrbracket}(s)| \mathcal{F}_s]
		\\
		& = E[X_s \bar{Y}^{+, n}_s \Ind_{\llbracket 0, \tau^n \rrbracket}(s)| \mathcal{F}_s]
		\\
		& = X_s E[\bar{Y}^{+, n}_s \Ind_{\llbracket 0, \tau^n \rrbracket}(s) | \mathcal{F}_s]
		= X_s Y^n_s, 
	\end{align*}
	where the fourth and fifth equalities are due to $X = Z^{X,n}$ on $\llbracket 0 , \tau^n \rrbracket$ and the inequality holds because $Z^{X,n}\bar{Y}^{+, n}\Ind_{\llbracket 0, \tau^n \rrbracket}$ is an $\mathbb{F}^{+, n}$-supermartingale as $Z^{X,n} \in \mathcal{Z}^{+, n}$; see Step 3.
	
	%	We want to show that $Y^n$ is an ($\mathbb{F}$-)SMD for $\mathcal{Z}^n := \left\{ X_{\tau^n \wedge .}\Ind_{[0,\infty)} + X_{\tau^n}\Ind_{[\infty]} : X \in \mathcal{X} \right\}$. So let $Z \in \mathcal{Z}^n$, then $ZY^n$ is clearly $\mathbb{F}$-adapted. Fix $s\leq t$, then by multiple application of the tower rule, we obtain
	%	\begin{align*}
	%	E\left[Z_t Y^n_t | \mathcal{F}_s\right] 
	%	& = E\left[Z_t E[ \bar{Y}^{+, n}_t \Ind_{\llbracket 0, \tau^n \rrbracket}(t)| \mathcal{F}_t] | \mathcal{F}_s\right]
	%	\\
	%	& = E\left[Z_t \bar{Y}^{+, n}_t \Ind_{\llbracket 0, \tau^n \rrbracket}(t) | \mathcal{F}_s\right]
	%	\\
	%	& = E\left[E[Z_t \bar{Y}^{+, n}_t \Ind_{\llbracket 0, \tau^n \rrbracket}(t) | \mathcal{F}^{+, n}_s] | \mathcal{F}_s\right]
	%	\\
	%	& \leq E\left[Z_s \bar{Y}^{+, n}_s \Ind_{\llbracket 0, \tau^n \rrbracket}(s)| \mathcal{F}_s\right]
	%	\\
	%	& = Z_s E\left[\bar{Y}^{+, n}_s \Ind_{\llbracket 0, \tau^n \rrbracket}(s) | \mathcal{F}_s\right]
	%	= Z_s Y^n_s, 
	%	\end{align*}
	%	where the inequality is due to the fact that $ZY^{+, n}$ is an $\mathbb{F}^{+, n}$-supermartingale because $Z~\in~\mathcal{Z}^n~\subset~\mathcal{Z}^{+, n}$.
	
	\textit{Step 5 (convergence):} Note that $\bar{Y}^{+, n}_t \Ind_{\llbracket 0, \tau^n \rrbracket}(t) \leq \bar{Y}^{+, n+1}_t \Ind_{\llbracket 0, \tau^{n+1} \rrbracket}(t)$ due to \eqref{eq 1} and $\tau^n \leq \tau^{n+1}$. Therefore $Y^n \leq Y^{n+1}$ \Pas{} and hence the sequence $(Y^n \Ind_{\llbracket 0, \widetilde{T} \llbracket})_{n\in\mathbb{N}}$ converges monotonically to $Y := \lim_{n \rightarrow \infty} Y^n \Ind_{\llbracket 0, \widetilde{T} \llbracket}$, which is well defined, adapted and takes values in $(0,\infty]$ on $\llbracket 0, \widetilde{T} \llbracket$. In particular, since $0$ is an absorbing state, we have $\{Y=0\} \subseteq \{X=0\}$ for each $X \in \mathcal{X}$ by Remark~\ref{rmk 2}. Moreover, for any $k\in\mathbb{N}$, the definition of $Y$, monotone convergence, $\Ind_{\{t<\widetilde{T}\}} \leq 1$ plus the fact that $\hat{X}^kY^n$ is a supermartingale for $\hat{X}^k\in\mathcal{X}$ and each $n\in\mathbb{N}$, and finally $X_0 = 1$ and $Y_0 =  E[\bar{Y}^{+, n}_0 \Ind_{\llbracket 0, \tau^n \rrbracket}(0)| \mathcal{F}_0] = 1$ yield
	$$
	E[\hat{X}_t^k Y_t] 
	= E\big[\hat{X}_t^k \lim_{n \rightarrow \infty} Y^n_t \Ind_{\{t<\widetilde{T}\}}\big] 
	= \lim_{n \rightarrow \infty} E[\hat{X}_t^k Y^n_t \Ind_{\{t<\widetilde{T}\}}] 
	\leq \lim_{n \rightarrow \infty}  E[\hat{X}_0^k Y^n_0 ]
	= 1.
	$$
	Therefore $Y$ is finite \Pas{} because $\hat{X}^k>0$ on $\llbracket 0, \tau^k \llbracket$, $\lim_{k\rightarrow \infty}\tau^k = \widetilde{T}$ and \mbox{$Y=0$} on $\llbracket \widetilde{T},\infty \llbracket$. It remains to show that $Y$ is an SMD for $\mathcal{X}$. So let \mbox{$0\leq s <t <\infty$,} \mbox{$A \in \mathcal{F}_s$} and \mbox{$X\in \mathcal{X}$}. On $\{ \widetilde{T}\leq s \}$, we have $Y_t = Y_s = 0$ by definition so it trivially follows that \mbox{$E[X_tY_t \Ind_{A \cap \{ \widetilde{T}\leq s \}}] \leq E[X_sY_s \Ind_{A \cap \{ \widetilde{T}\leq s \}}]$}. On $\{s <\widetilde{T}\}$ on the other hand, we have
	\begin{align*}
		E[X_tY_t \Ind_{A \cap \{s < \widetilde{T}\}}]
		& = E\big[0 \Ind_{A \cap \{s < \widetilde{T} \leq t\}} + X_t \lim_{n \rightarrow \infty}  Y^n_t \Ind_{A \cap \{t < \widetilde{T}\}}\big]
		\\
		& = \lim_{n \rightarrow \infty}  E[ X_t Y^n_t \Ind_{A \cap \{t < \widetilde{T}\}}] 
		\\
		& = \lim_{n \rightarrow \infty}  E[ X_t \bar{Y}^{+, n}_t \Ind_{\llbracket 0, \tau^n \rrbracket}(t) \Ind_{A \cap \{t < \widetilde{T}\}}] 
		\\
		& = \lim_{n \rightarrow \infty}  E[ Z^{X,n}_t \bar{Y}^{+, n}_t \Ind_{\llbracket 0, \tau^n \rrbracket}(t) \Ind_{A}] 
		\\
		& \leq \lim_{n \rightarrow \infty}  E[ Z^{X,n}_s \bar{Y}^{+, n}_s \Ind_{\llbracket 0, \tau^n \rrbracket}(s) \Ind_{A}] 
		\\
		& = \lim_{n \rightarrow \infty}  E[ X_s \bar{Y}^{+, n}_s \Ind_{\llbracket 0, \tau^n \rrbracket}(s) \Ind_{A \cap \{s < \widetilde{T}\}}] 
		\\
		& = \lim_{n \rightarrow \infty}  E[ X_s Y^n_s \Ind_{A \cap \{s < \widetilde{T}\}}] 
		\\
		& = E[ X_s Y_s \Ind_{A \cap \{s < \widetilde{T}\}}] 
	\end{align*}
	where the first equality is due to $X_t=Y_t=0$ on $\{\widetilde{T} \leq t\}$ and the definition of $Y$, the second uses monotone convergence, the third is by the definition of $Y^n$ and the fact that $X_t\Ind_{A \cap \{t < \widetilde{T}\}}$ is $\mathcal{F}_t$-measurable, the fourth uses $X = Z^{X,n}$ on $\llbracket 0 , \tau^n \rrbracket$ and $\{t \leq \tau^n\} \subseteq \{ t <\widetilde{T}\}$, the inequality holds because $Z^{X,n} \bar{Y}^{+, n} \Ind_{\llbracket 0, \tau^n \rrbracket}$ is an $\mathbb{F}^{+, n}$-supermartingale by Step 3. and $A \in \mathcal{F}_s \subseteq \mathcal{F}^{+, n}_s$, the fifth equality again uses $X = Z^{X,n}$ on $\llbracket 0 , \tau^n \rrbracket$ and $\{s \leq \tau^n\} \subset \{ s <\widetilde{T}\}$, the sixth the definition of $Y^n$ and the fact that $X_s\Ind_{A \cap \{s < \widetilde{T}\}}$ is $\mathcal{F}_s$-measurable, and the last is again due to monotone convergence. This completes the proof of \lq\lq (1) $\Rightarrow$ (2)\rq\rq.
	
	%	\begin{multline*}
	%	E[X_tY_t \Ind_{A \cap \{s < \widetilde{T}\}}]
	%	= E[0 \cdot \Ind_{A \cap \{s < \widetilde{T} \leq t\}} + X_t \lim_{n \rightarrow \infty}  Y^n_t \Ind_{A \cap \{t < \widetilde{T}\}}]
	%	\\
	%	= E[ X_t\lim_{n \rightarrow \infty}  Y^n_t \Ind_{A \cap \{t \leq \tau^n\}}] 
	%	= \lim_{n \rightarrow \infty}  E[X_{t \wedge \tau^n} Y^n_t \Ind_{A \cap \{t \leq \tau^n\}}] 
	%	\leq \lim_{n \rightarrow \infty} E[X_{t \wedge \tau^n} Y^n_t \Ind_{A \cap \{s < \tau^n\}}] 
	%	\\
	%	\leq \lim_{n \rightarrow \infty}  E[X_s Y^n_s \Ind_{A \cap \{s < \tau^n\}}]
	%	= E[\lim_{n \rightarrow \infty} X_s Y^n_s \Ind_{A \cap \{s < \tau^n\}}]
	%	= E[X_sY_s \Ind_{A \cap \{s < \widetilde{T}\}}],
	%	\end{multline*}
	%	where the first equality is because $X=Y=0$ on $\{\widetilde{T} \leq t\}$ and due to the definition of $Y$, the second equality is because $\lim_{n\rightarrow \infty}\tau^n = \widetilde{T}$ \Pas, the third equality is due to the monotone convergence theorem and the fact that $t = t \wedge \tau^n$ on $\{t \leq \tau^n\}$, the first inequality is due to $X, Y^n \geq 0$, the second inequality follows from the optional stopping theorem because $X_{\tau^n \wedge \cdot}Y^n$ is a supermartingale since $Y^n$ is a supermartingale deflator for $\mathcal{Z}^n$, $X_{\tau^n \wedge \cdot} \in \mathcal{Z}^n$ and $A \cap \{s < \widetilde{T} \} \in \mathcal{F}_s$ since $\widetilde{T}$ is a stopping time, the fourth equality is by the monotone convergence again, and finally the last equality again follows from the definition of $Y$. We obtain $E[X_tY_t \Ind_A]\leq E[X_sY_s \Ind_A]$, hence $Y$ is as desired.
	
	\lq\lq (2) $\Rightarrow$ (1)\rq\rq: If there exists a supermartingale deflator $Y$ which is adapted and has \mbox{$\{Y=0\} \subseteq \{X= 0\}$} for each $X \in \mathcal{X}$, then $0$ is an absorbing state by Remark~\ref{rmk 2}. Fix \mbox{$t\in[0,\infty)$}; we need to show that $\mathcal{X}_t$ is bounded in $L^0$, and this is standard. For each $M>0$, we have the inequality \mbox{$M \sup_{X \in \mathcal{X}} P[X_tY_t\geq M]\leq \sup_{X \in \mathcal{X}} E[X_t Y_t] \leq 1$} because $XY$ is a nonnegative supermartingale for each $X\in\mathcal{X}$. Therefore the set \mbox{$\{ X_tY_t: X \in \mathcal{X}\}$} is bounded in $L^0$. But $X_t = X_t\Ind_{\{Y_t >0\}}$ due to $\{Y=0\} \subseteq \{X= 0\}$ for each $X \in \mathcal{X}$ and hence we have \mbox{$\{ X_tY_t: X \in \mathcal{X}\} = \{ X_tY_t \Ind_{\{Y_t>0\}} : X \in \mathcal{X}\}$}. Therefore it follows that also the set \mbox{$\{ X_t \Ind_{\{Y_t  >0\}} : X \in \mathcal{X}\} = \{ X_t:  X \in \mathcal{X}\} = \mathcal{X}_t$} is bounded in $L^0$. This finishes the proof.
\end{proof}

%%%%%%%%%%%%%%%%%%%%%%%%%%%%%%%%%%%%%%%%%%%%%%%%%%%%%%%%%%%%%%%%%%%%%%%%%%%%%%%%%%%%%%%%%%%%%%%%

\section{Non-adapted framework}\label{appendix 2}

If the elements $X \in \mathcal{X}$ are not adapted (non-adapted framework), one needs a generalised notion of supermartingales. Kardaras \cite{Kar13d} defines a \textit{generalised supermartingale} as a process $Z$ which satisfies for all $0\leq s <t$ and $A \in \mathcal{F}_s$ that $E[Z_t/Z_s \Ind_A] \leq P(A)$. Note that $Z$ need not be adapted nor RC. We first remark that this is by no means similar to the requirement that the optional projection $Z^\mathbb{F}$ of $Z$ on $\mathbb{F}$ is a supermartingale, as the following example illustrates.

\begin{example}
	Consider the finite-time finite-space model $\Omega = \{a,b\}$ with the probability measure $P[a] = P[b] = 1/2$, and the trivial filtration defined by \mbox{$\mathcal{F}_t := \{ \emptyset, \Omega\}$} for $t = 0,1,2$. Let $Z(a)_{0,1,2} = (5, 6, 9)$ and $Z(b)_{0,1,2} = (5, 2, 1)$; then $Z$ is a generalised supermartingale, but  $Z^\mathbb{F}_{0,1,2} = E[Z_{0,1,2}]$ is not a supermartingale. On the other hand, let $W(a)_{0,1,2} = (5, 9, 6)$ and $W(b)_{0,1,2} = (5, 1, 2)$; then $W$ is not a generalised supermartingale, but $W^\mathbb{F}$ is a supermartingale.
\end{example}

We believe that the dual characterisation of DSV (Theorem~\ref{thm 1}) does not hold in the non-adapted framework without further adjustment. We base this belief on the following crucial differences between the adapted and non-adapted frameworks. In the non-adapted case, the statement of Lemma~\ref{lemma 5} below does not hold, i.e., if a generalised supermartingale is RC in probability, it need not be \Pas{} RCLL. Also, the minimum principle for nonnegative generalised supermartingales does not hold, which precludes argumentations as in the proof of Theorem~\ref{thm 1}, \lq\lq (1) $\Rightarrow$ (2)\rq\rq. That step is an essential requirement there to recover the asymptotic part of the dual condition (i.e., $\liminf_{t \upuparrows \hat{T}} (\hat{X}_t Y_t) >0 $ \Pas). In fact, we conjecture that there exists a strictly positive generalised supermartingale $Y$ which is RCLL in probability and such that $\liminf_{t \upuparrows \tau} Y_t = 0$ for some finite random time $\tau < \infty$.

We conjecture on the other hand that the dual characterisation of NUPBR$_{\rm loc}$, Theorem~\ref{thm 2} works also in the non-adapted framework. This, however, is a topic for future research and we do not pursue this question here.

%%%%%%%%%%%%%%%%%%%%%%%%%%%%%%%%%%%%%%%%%%%%%%%%%%%%%%%%%%%%%%%%%%%%%%%%%%%%%%%%%%%%%%%%%%%%%%%%

\begin{appendix}
	
\section{Auxiliary results}

\begin{lemma}\label{lemma 6}
	Let $S$, $\Theta$, $\mathcal{X}$ be as in Example~\ref{ex 2} and let $\eta \in \Theta$ be such that $\eta \geq0$, \mbox{$\eta \cdot S>0$},  $\eta \cdot S_- >0$ and $S/(\eta \cdot S)$ is bounded uniformly in $t\geq0$, \Pas{} (i.e., $\eta$ is a reference strategy in the sense of \cite[Definition~2.2]{BS18}. Then $\mathcal{X}$ satisfies DSV for $\hat{X}:=\eta\cdot S$ as in Definition~\ref{def dsv} if and only if $S$ satisfies DSV for $\eta$ as in \cite[Definition~2.7]{BS18}.
\end{lemma}

\begin{proof}
	Note that $\hat{T}  := \inf \{ t \geq 0: \hat{X}_t = 0 \} = \infty$ and hence DSV for $\hat{X}$ reads as saying that $A := \{ \liminf_{t \rightarrow \infty} (\theta_t\cdot S_t)(\eta_0 \cdot S_0) / \big((\eta_t \cdot S_t)(\theta_0 \cdot S_0) \big): \theta \in \Theta \}$ is bounded in  $L^0$. On the other hand, by \cite[Theorem~2.14]{BS18}, $S$ satisfies DSV for $\eta$ if and only if $S^\eta := S/(\eta \cdot S)$ satisfies NUPBR$_\infty$ as defined just before \cite[Theorem~2.14]{BS18}, meaning that 
	$$
	B:=\{ \lim_{t \rightarrow \infty} (\theta_t \cdot S^\eta_t) / (\theta_0 \cdot S^\eta_0): \theta \in \Theta, \lim_{t \rightarrow \infty} (\theta_t \cdot S^\eta_t) \text{ exists} \}
	$$ 
	is bounded in $L^0$. Hence we must show that $A$ is bounded in $L^0$ if and only if $B$ is.
	
	The \lq\lq only if\rq\rq{} direction follows immediately from the definitions of $A$, $B$ and $S^\eta$. For the \lq\lq if\rq\rq{} direction, we have that $S$ satisfies DSV for $\eta$ and claim that then $A = B$ \Pas{} Indeed, \cite[Theorem~2.11]{BS18}, implies that there exists a one-dimensional semimartingale $D>0$ with $D_- >0$ such that $S/D$ is a $\sigma$-martingale and $\inf_{t \in [0,\infty)} (\eta_t \cdot (S_t/D_t))>0$. Since $\theta \cdot (S/D) \geq0$ for any $\theta \in \Theta$, $\theta \cdot (S/D)$ is a supermartingale and so the convergence theorem implies that $\lim_{t \rightarrow \infty} (\theta_t \cdot (S_t/D_t))$ exists and is finite \Pas{} As $\eta \in \Theta$, we also have $0<\lim_{t \rightarrow \infty} \big(\eta_t \cdot (S_t/D_t)\big)$, and therefore for each $\theta \in \Theta$, the process \mbox{$\theta \cdot S^\eta = (\theta \cdot S)/(\eta \cdot S) = (\theta \cdot (S/D))/(\eta \cdot (S/D))$} has a limit at $\infty$. So $A = B$ \Pas, and this finishes the proof.
\end{proof}
	
\begin{lemma}\label{lemma 2}
	\begin{enumerate}[$\rm (a)$]
		\item If $\mathcal{X}$ is an SPD and $\hat{X}$ is as in \eqref{B}, then DSV for $\hat{X}$ implies NUPBR$_{\rm loc}$.
		\item If $\mathcal{X}$ is an SPP which has bounded time horizon $T$, then NUPBR$_{\rm loc}$ holds if and only if $\mathcal{X}_T$ is bounded in $L^0$.
		\item If $\mathcal{X}$ is an SPP which has bounded time horizon $T$ and $\bar{X}$ is as in \eqref{b}, then NUPBR$_{\rm loc}$ implies DSV for $\bar{X}$.
	\end{enumerate}
\end{lemma}

\begin{proof}
	(a) Let $\mathcal{X}$ be an SPD and $\hat{X}$ as in \eqref{B}. Suppose that NUPBR$_{\rm loc}$ is not satisfied and let $t \in [0,\infty)$ be such that $\mathcal{X}_t$ is not bounded in $L^0$. Then also $\mathcal{X}_t \Ind_{\{ t< \hat{T}\}}$ and hence $\mathcal{X}_t(\Ind_{\{ t< \hat{T}\}}/\hat{X}_t) := \{X_t\Ind_{\{ t< \hat{T}\}}/\hat{X}_t : X \in \mathcal{X}\}$ are not bounded in $L^0$. Fix $X \in \mathcal{X}$; then by \eqref{d}, $Y:= (X_{t \wedge \cdot}/\hat{X}_t) \hat{X}_{t \vee \cdot}$ is an element of $\mathcal{X}$. Note that \mbox{$\liminf_{s \upuparrows \hat{T}} (Y_s/\hat{X}_s) = X_t/\hat{X}_t$} on $\{t<\hat{T}\}$ and therefore we have $\liminf_{s \upuparrows \hat{T}} (Y_s/\hat{X}_s) \geq X_t\Ind_{\{ t< \hat{T}\}}/\hat{X}_t$. It follows that the set \mbox{$\{ \liminf_{s \upuparrows \hat{T}} (X_s/\hat{X}_s) : X\in \mathcal{X}\}$} is not bounded in $L^0$ either and hence DSV for $\hat{X}$ does not hold.
	
	(b) The \lq\lq only if\rq\rq{} direction is clear. For the \lq\lq if\rq\rq{} direction, let $\mathcal{X}$ be an SPP which has bounded time horizon $T$ and let $\bar{X}$ be as in \eqref{b}. Suppose further that $\mathcal{X}_T$ is bounded in $L^0$. Fix $t \in [0,T]$ and $X \in \mathcal{X}$; then by \eqref{d}, $Y:= (X_{t \wedge \cdot}/\bar{X}_t) \bar{X}_{t \vee \cdot}$ is an element of $\mathcal{X}$. Since $Y_T = (X_t/\bar{X}_t)\bar{X}_T$, we have $\mathcal{X}_t \subseteq \mathcal{X}_T (\bar{X}_t/\bar{X}_T)$ and due to $\bar{X}>0$ \Pas, it follows that $\mathcal{X}_t$ is bounded in $L^0$ as well. Since $\mathcal{X}_t = \mathcal{X}_T$ for every $t\geq T$, we conclude.
	
	(c) Let $\mathcal{X}$ be an SPP which has bounded time horizon $T$ and $\bar{X}$ as in \eqref{b}. Then \mbox{$\bar{T} := \inf\{ t\geq0: \bar{X}_t = 0 \} = +\infty$} \Pas{} and for each $X \in \mathcal{X}$, \mbox{$\liminf_{s \upuparrows \bar{T}} (X_s/\bar{X}_s) = X_T/\bar{X}_T$}. Since $0<\bar{X}_T<\infty$ \Pas, it follows that \mbox{$\{ \liminf_{s \upuparrows \bar{T}} (X_s/\hat{X}_s) : X\in \mathcal{X}\} = \mathcal{X}_T/\bar{X}_T$} is bounded in $L^0$ if and only if $\mathcal{X}_T$ is, and hence we conclude by (b).
\end{proof}

For any function $z : [0, \infty) \rightarrow \mathbb{R}$, set $\underline{z}(\infty) := \liminf_{t \rightarrow \infty} z(t)$.

\begin{lemma}\label{lemma liminf}
	Let $x,y : [0,\infty) \rightarrow \mathbb{R}$ be nonnegative. Then $\underline{x}(\infty)  \underline{y}(\infty) \leq \underline{xy}(\infty)$ holds under each of the following conditions:
	\begin{enumerate}[$\rm (a)$]
		\item Both $\underline{x}(\infty) < \infty$ and $\underline{y}(\infty)  < \infty$.
		\item $\underline{x}(\infty) = \infty$ and $\underline{y}(\infty) > 0$.
		\item $0<\underline{y}(\infty) <\infty$.
	\end{enumerate}
\end{lemma}

\begin{proof}
	For (a), suppose to the contrary that $c := \underline{x}(\infty)  \underline{y}(\infty) - \underline{xy}(\infty)>0$. Since $x,y$ are nonnegative, $\underline{xy}(\infty)\geq0$ and hence it follows that $\underline{x}(\infty), \underline{y}(\infty) >0$. Take $\epsilon>0$ such that $\epsilon < \min (\underline{x}(\infty), \underline{y}(\infty), c/(\underline{x}(\infty) + \underline{y}(\infty))$. By definition, there exists a $T<\infty$ such that we have $x(t) \geq \underline{x}(\infty) -\epsilon >0$ and $y(t) \geq \underline{y}(\infty) - \epsilon >0$ for all $t \geq T$. Hence for all $t\geq T$, $x(t)  y(t) > \underline{x}(\infty) \underline{y}(\infty) - \epsilon(\underline{x}(\infty) + \underline{y}(\infty))$ and therefore $\underline{xy}(\infty) > \underline{x}(\infty) \underline{y}(\infty) - c$ which is a contradiction.
	
	For (b), note that $\underline{x}(\infty) = \infty$ means that $\lim_{t \rightarrow \infty} x(t) = \infty$. Also, $\underline{y}(\infty) > 0$ implies that there exist a $T<\infty$ and a $c>0$ such that $y(t) \geq c$ for all $t \geq T$. It follows that $\lim_{t \rightarrow \infty} x(t) y(t) = \infty$. 
	
	For (c), if $\underline{x}(\infty) < \infty$, then the statement follows from (a), whereas if \mbox{$\underline{x}(\infty) = \infty$}, then the statement follows from (b). This finishes the proof.
\end{proof}

\begin{remark}
	\textbf{1)} For the sake of completeness, we remark that if $\underline{x}(\infty) = \infty$ and $\underline{y}(\infty) = 0$, then $\underline{xy}(\infty)$ can take any value in $[0,\infty]$. Hence depending on the convention for $\infty \cdot 0$, one might or might not assert $\underline{x}(\infty) \underline{y}(\infty) \leq \underline{xy}(\infty)$ in this case.
	
	\textbf{2)} If we replace $[0,\infty)$ by $[0,T)$ for some $T<\infty$ and $\liminf_{t \rightarrow \infty}$ by $\liminf_{t \upuparrows T}$, analogous results hold as well.
\end{remark}

\begin{lemma}\label{lemma 5}
	Let $X$ be a nonnegative supermartingale on $[0,\infty)$ which is RC in probability, and let the underlying filtration satisfy the usual conditions. Then $X$ has a \Pas{} RCLL modification.
\end{lemma}

\begin{proof}
	By Protter \cite[Chapter I, Theorem~9]{Pro12}, it is sufficient to show that the function $t \mapsto E[X_t]$ is right-continuous. Fix $t \in [0,\infty)$ and let $(t^n)_{n \in \mathbb{N}}$ be a sequence in $[0,\infty)$ with $t^n \downarrow t$; then for each $n \in \mathbb{N}$, $E[X_{t^n}]\leq E[X_t]$ and since $n \mapsto E[X_{t^n}]$ is increasing because $X$ is a supermartingale, it converges. By the right-continuity of $X$ in probability, there is an $N$ for any $\epsilon>0$ such that for all $n\geq N$, we have $P[|X_t - X_{t^n}| >\epsilon]<\epsilon$. It follows due to $X\geq 0$ that \mbox{$E[X_{t^n}] \geq E[X_t] - \epsilon - \sup\{ E[X_t \Ind_A] : P[A]<\epsilon \}$} for $n\geq N$. But since $X_t$ is integrable, $\lim_{\epsilon \rightarrow 0} \sup\{ E[X_t \Ind_A] : P[A]<\epsilon \} = 0$ and hence $\lim_{n \rightarrow \infty} E[X_{t^n}]\geq E[X_t]$. It follows that $\lim_{n \rightarrow \infty}E[X_{t^n}] = E[X_t]$ and hence we conclude.
\end{proof}

\begin{remark}
	For completeness, we remark that Lemma~\ref{lemma 5} holds also for not necessarily nonnegative $X$.
\end{remark}

\begin{example}\label{ex 3}
	We present a simple deterministic example for an SPD $\mathcal{X}$ with $\hat{X}$ as in \eqref{B} which satisfies NUPBR$_{\rm loc}$, but does not satisfy DSV for $\hat{X}$. Consider the interval $[0,\infty)$ and let $\hat{X} := \Ind_{[0,1)}$ and $\tilde{X} := \sum_{n = 1}^\infty n \Ind_{[(n-1)/n, n/(n+1))}$. Let $\mathcal{X}$ be the smallest set containing $\hat{X}$ and $\tilde{X}$ such that \eqref{c} and \eqref{d} are satisfied. Then $\mathcal{X}$ is an SPD with $\hat{T} \equiv 1$. It is straightforward to check by induction that for each $X \in \mathcal{X}$, $X \leq \tilde{X}$ and hence $\mathcal{X}_t$ is bounded in $L^0$ for all $t \in [0,\infty)$. On the other hand, $\liminf_{t \upuparrows \hat{T}} (\tilde{X}_t / \hat{X}_t) = \infty$ and hence DSV for $\hat{X}$ does not hold.
\end{example}

\begin{lemma} \label{lemma 7}
	Let $\mathcal{A}$ be a $\sigma$-algebra on $\Omega$ and $\mathcal{V} = \{ V_1, \dots, V_K\}$ for some $K \in \mathbb{N}$ a finite partition of $\Omega$ into pairwise disjoint (not necessarily $\mathcal{A}$-measurable) sets. Then 
	$$
	\mathcal{A} \vee \sigma(\mathcal{V}) = \sigma(\mathcal{A} \cup \mathcal{V}) = \bigg\{\bigcup_{k = 1}^K (V_k \cap A_k): A_k \in \mathcal{A}, k = 1, \dots, K \bigg\} =: \mathcal{A}'.
	$$
\end{lemma}

\begin{proof}
	First we verify that $\mathcal{A}'$ is a $\sigma$-algebra. It is clear that $\emptyset \in \mathcal{A}'$, and for any sequence $(B_j)_{j \in \mathbb{N}} \subseteq \mathcal{A}'$, each $B_j$ can be written as $B_j = \bigcup_{k = 1}^K (V_k \cap A_{j, k})$ for some $A_{j,k} \in \mathcal{A}$, \mbox{$k = 1, \dots, K$}, and hence $\bigcup_{j\in\mathbb{N}} B_j =  \bigcup_{k = 1}^K (V_k \cap \bigcup_{j\in\mathbb{N}} A_{j,k}) \in \mathcal{A}'$. It remains to show that for any $B\in\mathcal{A}'$, $B^c$ is an element of $\mathcal{A}'$. So let $B \in \mathcal{A}'$; then $B = \bigcup_{k = 1}^K (V_k \cap A_{k})$ for some $A_k \in \mathcal{A}$, $k = 1, \dots, K$. Because $V_1, \dots, V_K$ form a pairwise disjoint partition of $\Omega$, we have $B^c = \bigcup_{k = 1}^K (V_k \cap B^c)$, and we claim that 
	\begin{equation}\label{eq 2}
		V_k \cap B^c = V_k \cap A^c_k \mspace{20mu} \text{for } k = 1, \dots, K.
	\end{equation}
	This then implies that $B^c = \bigcup_{k = 1}^K (V_k \cap A_k^c)$ is in $\mathcal{A}'$. To argue \eqref{eq 2}, fix $k$ and note that $V^c_k = \bigcup_{j \neq k} V_j$ because $V_1, \dots, V_K$ form a pairwise disjoint partition of $\Omega$. Then we can write
	$$
	V_k \cap A^c_k
	= V_k \cap \big(V_k \cap (V^c_k \cup A^c_k)\big)
	= V_k \cap \big( (V_k \cap A_k) \cup V^c_k\big)^c
	= V_k \cap \big( (V_k \cap A_k) \cup \bigcup_{j \neq k} V_j\big)^c.
	$$
	Note that $\bigcap_{j \neq k} V_j \supseteq \bigcap_{j \neq k} (V_j \cap A_j)$ so that \mbox{$V_k \cap A_k^c \subseteq V_k \cap (\bigcup_{j=1}^K (V_j \cap A_j))^c = V_k \cap B^c$}. On the other hand, we also have of course that $\bigcup_{j=1}^K (V_j \cap A_j) \supseteq V_k \cap A_k$ so that we get \mbox{$V_k \cap B^c \subseteq V_k \cap (V_k \cap A_k)^c = V_k \cap A_k^c$}. This proves \eqref{eq 2}.
	
	So $\mathcal{A}'$ is a $\sigma$-algebra and hence $\mathcal{A} \subseteq \mathcal{A}'$ and $\mathcal{V} \subseteq \mathcal{A}'$ yield \lq\lq $\subseteq$\rq\rq. Finally, \lq\lq $\supseteq$\rq\rq{} is clear since every $\sigma$-algebra containing $\mathcal{V}$ and $\mathcal{A}$ contains all elements of the form $\bigcup_{k = 1}^K (V_k \cap A_k)$ for $A_k \in \mathcal{A}, k = 1, \dots, K$.
\end{proof}

\begin{prop}\label{char t}
	Let $\mathcal{X}$ be an SP and $\widetilde{T}$ as in Proposition~\ref{prop 1}. Suppose $\tilde{\tau}$ is a positive random time such that for each $s < t$ with $P[\tilde{\tau} \in (s, t]] > 0$, there exists an $\hat{X}^{(s,t)} \in \mathcal{X}$ such that $\hat{X}^{(s,t)}$ is positive on $[0,s]$ $Q^{(s,t)}$-a.s., and each $X \in \mathcal{X}$ is zero on $[t, \infty)$  $Q^{(s,t)}$-a.s., where $Q^{(s,t)} := P[\thinspace\cdot \thinspace| \tilde{\tau} \in (s, t]]$. Then $\tilde{\tau} = \widetilde{T}$ \Pas
\end{prop}

\begin{proof}
	Suppose first that $P[\tilde{\tau} < \widetilde{T}]>0$ and let $t \in (0,\infty)$ be such that \mbox{$P[\tilde{\tau} < t < \widetilde{T}]>0$}. Then $P[\tilde{\tau} \in (0, t]] > 0$ and hence each $X \in \mathcal{X}$ is zero on $[t, \infty)$  $Q^{(0,t)}$-a.s. But by \mbox{Proposition~\ref{prop 1}}, there exists a sequence $(\hat{X}^n)_{n \in \mathbb{N}} \subseteq \mathcal{X}$ with \mbox{$\hat{T}^n := \inf\{t\geq0: \hat{X}^n=0\}$} such that $\hat{T}^n \nearrow \widetilde{T}$ as $n \rightarrow \infty$. So there is $n \in \mathbb{N}$ such that $\hat{T}^n>t$ holds with positive $Q^{(0,t)}$-probability, which is a contradiction. Suppose now that $P[\widetilde{T} < \tilde{\tau}]>0$ and let $t \in (0,\infty)$ be such that $P[\widetilde{T} <t < \tilde{\tau}]>0$. Then $P[\tilde{\tau} \in (t, \infty]] > 0$ and hence there exists an $\hat{X}^{(t,\infty)}$ which is positive on $[0, t]$ $Q^{(t,\infty)}$-a.s. Proposition~\ref{prop 1} on the other hand implies that for each $X  \in \mathcal{X}$, $X_t\Ind_{\{\widetilde{T}\leq t\}} = 0$ on $\{\tilde{\tau} \in (t,\infty]\}$ \Pas, and in particular $\hat{X}^{(t,\infty)}_t \Ind_{\{\widetilde{T}\leq t\}} = 0$ holds $Q^{(t,\infty)}$-a.s. due to $Q^{(t, \infty)} \ll P$. But since $P[\{ \widetilde{T} \leq t\} \cap \{ t < \tilde{\tau} \}] >0$, we have $Q^{(t, \infty)}[ \widetilde{T} \leq t]>0$, which is a contradiction to the $Q^{(t,\infty)}$-a.s. positivity of $\hat{X}^{(t,\infty)}$ on $[0,t]$. Therefore $P[\tilde{\tau} < \widetilde{T}] = P[\widetilde{T} < \tilde{\tau}] = 0$ and hence $\tilde{\tau} = \widetilde{T}$ \Pas{}
\end{proof}
	
\end{appendix}

\section*{Acknowledgements}

The author thanks Martin Schweizer for multiple thorough proofreading and many invaluable comments, remarks and suggestions on initial versions of this paper. The author also thanks Robert Crowell for linguistic support, and Matteo Burzoni, \mbox{Andrea Gabrielli}, Florian Krach, Chong Liu, David Martins and Ariel Neufeld for discussions and general support. 

Financial support by the ETH Foundation via the Stochastic Finance Group (SFG) at ETH Zurich and by the Swiss Finance Institute (SFI) is gratefully acknowledged.

\bibliographystyle{amsplain}
\frenchspacing
\bibliography{bib}

\end{document}